\documentclass[12pt, reqno]{amsart}
\usepackage{amsmath, amsthm, amscd, amsfonts, amssymb, graphicx, color}
\usepackage[bookmarksnumbered, colorlinks, plainpages]{hyperref}
\hypersetup{colorlinks=true,linkcolor=red, anchorcolor=green, citecolor=cyan, urlcolor=red, filecolor=magenta, pdftoolbar=true}

\newtheorem{theorem}{Theorem}[section]
\newtheorem{lemma}[theorem]{Lemma}

\newtheorem{corollary}[theorem]{Corollary}
\theoremstyle{definition}
\newtheorem{definition}[theorem]{Definition}
\newtheorem{example}[theorem]{Example}

\theoremstyle{remark}

\numberwithin{equation}{section}

\begin{document}
	\setcounter{page}{1}
	
	\title[C-CLASS FUNCTIONS AND FIXED POINTS FOR WEAKLY CONTRACTIVE MAPPINGS]{C-CLASS FUNCTIONS AND FIXED POINTS FOR WEAKLY CONTRACTIVE MAPPINGS IN
		RECTANGULAR b-METRIC SPACES}
	
	\author[M. Rossafi, A. Kari, A. H. Ansari]{Mohamed
		Rossafi$^{1*}$, Abdelkarim Kari$^{2}$ \MakeLowercase{and} Arsalan Hojjat Ansari$^{3}$}

		\address{$^{1}$LaSMA Laboratory Department of Mathematics Faculty of Sciences Dhar El Mahraz, University Sidi Mohamed Ben Abdellah, Fez, Morocco}
	\email{\textcolor[rgb]{0.00,0.00,0.84}{rossafimohamed@gmail.com}}
	
	\address{$^{2}$ Laboratory of Analysis, Modeling and Simulation Faculty of Sciences Ben M'Sik, Hassan II University, B.P. 7955
		Casablanca, Morocco}
	\email{\textcolor[rgb]{0.00,0.00,0.84}{abdkrimkariprofes@gmail.com}}
	
	\address{$^{3}$Department of Mathematics and Applied Mathematics, Sefako
		Makgatho Health Sciences University, Ga-Rankuwa, Pretoria, Medunsa-0204,
		South Africa}
	\email{\textcolor[rgb]{0.00,0.00,0.84}{mathanalsisamir4@gmail.com}}
	
	
	
	\date{%10/03/2020; %Accepted: zzzzzz.
		\newline \indent $^{*}$Corresponding author}
	
	\begin{abstract}
In this paper, inspired by the concept of generalized weakly contractive mappings in metric spaces, we introduce C-Class function and fixed point theory for weakly contractive in the setting of rectangular $b$-metric spaces and established existence and uniqueness of fixed points for the mappings introduced.
		\end{abstract} \maketitle

\section{Introduction}

It is well known that the Banach contraction principle \cite{BA} is a
fundamental result in the fixed point theory, several authors have obtained
many interesting extensions and generalizations \cite{BRO, KAN, KRMA, RE}.
The well known metric spaces have been generalized metric spaces introduced
by Branciari \cite{BRA}. Various fixed point results were established on
such spaces \cite{AZ, JS, JSA, KIR, LKS, SM}.

Recently, George \textit{et al. } \cite{RG} announced the notion of $b$
-rectangular metric space and formulated some fixed point theorems in the $b$
-rectangular metric space. Many authors initiated and studied many existing
fixed point theorems in such spaces \cite{DINI, DIN, c, KARO, KAR, lhl, a,
	b, d, RO, SK, SKS, ysg}.

Weak contraction principle is a generalization of the Banach contraction
principle which was first given by Alber \textit{et al.} in Hilbert spaces 
\cite{YA}. Coudhury \textit{et al.} \cite{CHOD} proved some fixed point
results for weakly contractive mappings in complete metric spaces. Several
authors have studied weak contraction mapping in complete metric spaces \cite%
{CHI, KHAN, mb, nia, prs, rsa, ROD, ZH}.

Very recently, Cho \cite{CHO} introduced a special weakly contractive
mappings called generalized weakly contractive mappings and proved some
fixed point results for such mappings in complete metric spaces.

In this work, we introduce a new notion of C-Class function for weakly
contractive mappings and provide some fixed point results for such mappings
in complete $b$-rectangular metric spaces. We also present some special
examples of generalized weakly contractive mappings on $b$-rectangular
metric spaces. Also, we derive some useful corollaries.

\section{Preliminaries}

In the following, we collect background information needed in the
presentation of our results.

\begin{definition}
	\cite{RG} Let $X$ be a nonempty set, $s\geq 1$ be a given real number and $d 
	$: $X\times X\rightarrow \left[ 0,+\infty \right[ $ be a function such that
	for all $x,y$ $\in X$ and all distinct points $u,v\in X,$
	
	\begin{itemize}
		\item[1.] $d\left(x, y\right) =0$ if only if $x=y;$
		
		\item[2.] $d\left(x, y\right) =d\left(y, x\right);$
		
		\item[3.] $d\left(x, y\right) \leq $ $s\left[d\left( x,u\right)+d\left(u,
		v\right) +d\left(v, y\right) \right] $ $\left(b-rectangular\ inequality
		\right) .$
	\end{itemize}
	
	Then $\left(X, d\right) $ is called a $b$-rectangular metric space.
\end{definition}

\begin{example}
	\cite{KARO}. Let $X=A\cup B $, where $A=\lbrace \frac{1}{n}:n\in\lbrace
	2,3,4,5,6,7\rbrace \rbrace $ and $B=\left[1,2 \right] $. Define $d:X\times
	X\rightarrow \left[0,+\infty \right[ $ as follows: 
	\begin{equation*}
		\left\lbrace \begin{aligned} d(x, y) &=d(y, x)\ for \ all \  x,y\in X;\\
			d(x, y) &=0\Leftrightarrow y= x\\ \end{aligned} \right.
	\end{equation*}
	and 
	\begin{equation*}
		\left\lbrace \begin{aligned} d\left( \frac{1}{2},\frac{1}{3}\right) =d\left(
			\frac{1}{4},\frac{1}{5}\right) =d\left( \frac{1}{6},\frac{1}{7}\right)
			&=0,05\\ d\left( \frac{1}{2},\frac{1}{4}\right) =d\left(
			\frac{1}{3},\frac{1}{7}\right) =d\left( \frac{1}{5},\frac{1}{6}\right)
			&=0,08\\ d\left( \frac{1}{2},\frac{1}{6}\right) =d\left(
			\frac{1}{3},\frac{1}{4}\right) =d\left( \frac{1}{5},\frac{1}{7}\right)
			&=0,4\\ d\left( \frac{1}{2},\frac{1}{5}\right) =d\left(
			\frac{1}{3},\frac{1}{6}\right) =d\left( \frac{1}{4},\frac{1}{7}\right)
			&=0,24\\ d\left( \frac{1}{2},\frac{1}{7}\right) =d\left(
			\frac{1}{3},\frac{1}{5}\right) =d\left( \frac{1}{4},\frac{1}{6}\right)
			&=0,15\\ d\left( x,y\right) =\left( \vert x-y\vert\right) ^{2} \ otherwise.
		\end{aligned} \right.
	\end{equation*}
	Then $(X,d) $ is a $b$-rectangular metric space with coefficient $s=3$.
\end{example}

\begin{lemma}
	\label{2.3} \cite{RO} Let $\left(X, d\right) $ be a $b$-rectangular metric
	space.
	
	\begin{itemize}
		\item[(a)] Suppose that sequences $\lbrace x_{n}\rbrace $ and $\lbrace
		y_{n}\rbrace $ in $X$ are such that $x_{n}\rightarrow x$ and $%
		y_{n}\rightarrow y$ as $n\rightarrow +\infty ,$ with $x\neq y,$ $x_{n}\neq x$
		and $y_{n}\neq y$ for all $n\in \mathbb{N}.$ Then we have 
		\begin{equation*}
			\frac{1}{s}d\left( x,y\right) \leq \lim_{n\rightarrow +\infty }\inf
			d\left(x_{n},y_{n}\right) \leq \lim_{n\rightarrow +\infty }\sup
			d\left(x_{n},y_{n}\right) \leq sd\left( x,y\right). \\
		\end{equation*}
		
		\item[(b)] If $y\in X$ and $\lbrace x_{n}\rbrace $ is a Cauchy sequence in $%
		X $ with $x_{n}\neq x_{m}$ for any $m,n\in \mathbb{N},$ $m\neq n,$
		converging to $x\neq y,$ then 
		\begin{equation*}
			\frac{1}{s}d\left( x,y\right) \leq \lim_{n\rightarrow +\infty }\inf
			d\left(x_{n},y\right) \leq \lim_{n\rightarrow +\infty }\sup d\left(
			x_{n},y\right)\leq sd\left( x,y\right),
		\end{equation*}
		for all $x\in X.$\newline
	\end{itemize}
\end{lemma}

\begin{lemma}
	\label{2?4} \cite{KARO} Let $\left(X, d\right) $ be a $b$-rectangular metric
	space and $\lbrace x_n \rbrace $ be a sequence in $X $ such that 
	\begin{equation}
		\lim_{n\rightarrow +\infty } d\left(x_{n},x_{n+1}\right)= \lim_{n\rightarrow
			+\infty }\ d\left(x_{n},x_{n+2}\right)=0.
	\end{equation}
	If $\lbrace x_n \rbrace $ is not a Cauchy sequence, then there exist $%
	\varepsilon >0 $ and two sequences $\lbrace m(k) \rbrace $ and $\lbrace n(k)
	\rbrace $ of positive integers such that \newline
	\begin{equation*}
		\varepsilon \leq \lim_{k\rightarrow +\infty }\inf d\left( x_{m_{\left(
				k\right) }},x_{n_{\left( k\right)}}\right) \leq \lim_{k\rightarrow +\infty
		}\sup d\left( x_{m_{\left( k\right) }},x_{n_{\left( k\right)}}\right)\leq
		s\varepsilon ,
	\end{equation*}
	\begin{equation*}
		\varepsilon \leq \lim_{k\rightarrow +\infty }\inf d\left( x_{n_{\left(
				k\right) }},x_{m_{\left( k\right)+1}}\right) \leq \lim_{k\rightarrow +\infty
		}\sup d\left( x_{n_{\left( k\right) }},x_{m_{\left( k\right)+1}}\right)\leq
		s\varepsilon ,
	\end{equation*}
	\begin{equation*}
		\varepsilon \leq \lim_{k\rightarrow +\infty }\inf d\left( x_{m_{\left(
				k\right) }},x_{n_{\left( k\right)+1}}\right) \leq \lim_{k\rightarrow +\infty
		}\sup d\left( x_{m_{\left( k\right) }},x_{n_{\left( k\right)+1}}\right)\leq
		s\varepsilon ,
	\end{equation*}
	\begin{equation*}
		\frac{\varepsilon}{s} \leq \lim_{k\rightarrow +\infty }\inf d\left(
		x_{m_{\left( k\right)+1 }},x_{n_{\left( k\right)+1}}\right) \leq
		\lim_{k\rightarrow +\infty }\sup d\left( x_{m_{\left( k\right)+1
		}},x_{n_{\left( k\right)+1}}\right)\leq s^{2}\varepsilon .
	\end{equation*}
\end{lemma}

\begin{definition}
	A function $f : X \rightarrow\mathbb{R^{+}} $, where $X $ is a $b$
	-rectangular metric space, is called lower semicontinuous if for all $x \in
	X $ and $x _{n }\in X$ with $lim_{n\rightarrow +\infty } x _{n }=x $, we
	have 
	\begin{equation*}
		f\left(x \right)\leq \liminf_{n\rightarrow +\infty }f\left( x _{n } \right).
	\end{equation*}
\end{definition}

\begin{definition}
	A function $g : X \rightarrow\mathbb{R^{+}} $, where $X $ is a $b$
	-rectangular metric space, is called is a right upper semicontinuous
	function if for all $x \in X $ and $x _{n }\in X$ with $lim_{n\rightarrow
		+\infty } x _{n }=x $, we have 
	\begin{equation*}
		g\left(x \right)\geq \limsup_{n\rightarrow +\infty }f\left( x _{n } \right).
	\end{equation*}
\end{definition}

\begin{definition}
	\cite{KHAN} A function $\psi:\left[0,+\infty \right[\rightarrow \left[
	0,+\infty \right[ $ is said to be an altering distance function if it
	satisfies the following conditions:
	
	\item[(a)] is continuous and nondecreasing;
	
	\item[(b)] $\psi (t)=0$ if and only if $t=0$.
\end{definition}

We denote the set of altering distance functions by $\Phi .$

\begin{example}
	Define $\psi _{1}$; $\psi _{2}$; $\psi _{3}$: $\left[ 0,+\infty \right[
	\rightarrow \left[ 0,+\infty \right[ $ by $\psi _{1}(t)=t$, $\psi _{t}(t)=2t$
	and $\psi _{3}(t)=t^{2}$. Then they are altering distance functions.
\end{example}

\begin{definition}
	\cite{CHO} Let $X$ be a complete metric space with metric $d$, and $%
	T:X\rightarrow X$. Also let $\varphi :X\rightarrow \mathbb{R^{+}}$ be a
	lower semicontinuous function. Then $T$ is called a generalized weakly
	contractive mapping if it satisfies the following condition: 
	\begin{equation*}
		\psi \left( d\left( Tx,Ty\right) +\varphi \left( Tx\right) +\varphi \left(
		Ty\right) \right) \leq \psi \left( m\left( x,y,d,T,\varphi \right) \right)
		-\phi \left( l\left( x,y,d,T,\varphi \right) \right) ,
	\end{equation*}
	where 
	\begin{equation*}
		m(x,y,d,T,\varphi )=\max \{d(x,y)+\varphi (x)+\varphi (y),d(x,Tx)+\varphi
		(x)+\varphi (Tx),d(y,Ty)+\varphi (y)+\varphi (Ty),
	\end{equation*}
	\begin{equation*}
		\frac{1}{2}\{d(x,Ty)+\varphi (x)+\varphi (Ty)+d(y,Tx)+\varphi (Tx)+\varphi
		(y)\}\}
	\end{equation*}
	and 
	\begin{equation*}
		l(x,y,d,T,\varphi )=\max \{d(x,y)+\varphi (x)+\varphi (y),d(y,Ty)+\varphi
		(y)+\varphi (Ty)\}
	\end{equation*}
	for all $x,y\in X$, where $\psi :\mathbb{R^{+}}\rightarrow \mathbb{R^{+}}$
	is continuous with $\psi (t)=0$ if and only if $t=0$ and $\phi :\mathbb{%
		R^{+} }\rightarrow \mathbb{R^{+}}$ is a lower semicontinuous function with $%
	\phi (t)=0$ if and only if $t=0$.
\end{definition}

\begin{theorem}
	\cite{CHO} Let $X$ be complete. If $T$ is a generalized weakly contractive
	mapping, then there exists a unique $z\in X$ such that $z=Tz$ and $\varphi
	(z)=0$.
\end{theorem}

In 2014 the concept of $C$-class functions was introduced by H. Ansari in 
\cite{aha} .

\begin{definition}
	\label{C-class}\cite{aha}A mapping $F:[0,\infty )^{2}\rightarrow \mathbb{R}$
	is called \textit{$C$-class} function if it is continuous and satisfies
	following axioms:
	
	(1) $F(s,t)\leq s$;
	
	(2) $F(s,t)=s$ implies that either $s=0$ or $t=0$; for all $s,t\in \lbrack
	0,\infty )$.
\end{definition}

Note for some $F$ we have that $F(0,0)=0$.

We denote $C$-class functions \ as $\mathcal{C}$.

\begin{example}
	\label{C-class examp}\cite{aha}The following functions $F:[0,\infty
	)^{2}\rightarrow \mathbb{R}$ are elements of $\mathcal{C}$, for all $s,t\in
	\lbrack 0,\infty )$:
	
	(1) $F(s,t)=s-t$, $F(s,t)=s\Rightarrow t=0$;
	
	(2) $F(s,t)=ms$, $0{<}m{<}1$, $F(s,t)=s\Rightarrow s=0$;
	
	(3) $F(s,t)=\frac{s}{(1+t)^{r}}$; $r\in (0,\infty )$, $F(s,t)=s$ $%
	\Rightarrow $ $s=0$ or $t=0$;
	
	(4) $F(s,t)=\log (t+a^{s})/(1+t)$, $a>1$, $F(s,t)=s$ $\Rightarrow $ $s=0$ or 
	$t=0$;
	
	(5) $F(s,t)=\ln (1+a^{s})/2$, $a>e$, $F(s,1)=s$ $\Rightarrow $ $s=0$;
	
	(6) $F(s,t)=(s+l)^{(1/(1+t)^{r})}-l$, $l>1,r\in (0,\infty )$, $F(s,t)=s$ $%
	\Rightarrow $ $t=0$;
	
	(7) $F(s,t)=s\log _{t+a}a$, $a>1$, $F(s,t)=s\Rightarrow $ $s=0$ or $t=0$;
	
	(8) $F(s,t)=s-(\frac{1+s}{2+s})(\frac{t}{1+t})$, $F(s,t)=s\Rightarrow t=0$;
	
	(9) $F(s,t)=s\beta (s)$, $\beta :[0,\infty )\rightarrow (0,1)$,and is
	continuous, $F(s,t)=s\Rightarrow s=0$;
	
	(10) $F(s,t)=s-\frac{t}{k+t},F(s,t)=s\Rightarrow t=0$;
	
	(11) $F(s,t)=s-\varphi (s),F(s,t)=s\Rightarrow s=0,$here $\varphi :[0,\infty
	)\rightarrow \lbrack 0,\infty )$ is a continuous function such that $\varphi
	(t)=0\Leftrightarrow t=0$;
	
	(12) $F(s,t)=sh(s,t),F(s,t)=s\Rightarrow s=0,$here $h:[0,\infty )\times
	\lbrack 0,\infty )\rightarrow \lbrack 0,\infty )$is a continuous function
	such that $h(t,s)<1$ for all $t,s>0$;
	
	(13) $F(s,t)=s-(\frac{2+t}{1+t})t$, $F(s,t)=s\Rightarrow t=0$.
\end{example}

(14) $F(s,t)=\sqrt[n]{\ln (1+s^{n})}$, $F(s,t)=s\Rightarrow s=0$.

(15) $F(s,t)=\phi (s),F(s,t)=s\Rightarrow s=0,$here $\phi :[0,\infty
)\rightarrow \lbrack 0,\infty )$ is a upper semicontinuous function such
that $\phi (0)=0,$ and $\phi (t)<t$ for $t>0,$

(16)$F(s,t)=\frac{s}{(1+s)^{r}}$; $r\in (0,\infty )$, $F(s,t)=s$ $%
\Rightarrow $ $s=0$ ;

\begin{definition}
	\cite{aha}Let $\Phi _{u}$ denote the class of the functions $\varphi
	:[0,\infty )\rightarrow \lbrack 0,\infty )$ which satisfy the following
	conditions:
\end{definition}

\begin{enumerate}
	\item[$(a)$] $\varphi $ is continuous ;
	
	\item[$(b)$] $\varphi (t)>0,t>0$ \ and $\varphi (0)\geq 0$\ .
\end{enumerate}

\bigskip

\begin{definition}
	\label{3.1 +} \cite{mdk}Let $X$ be a complete $b$-rectangular metric space
	with metric $d$ and parameter $s$ and $T:X\rightarrow X$. Also let $\varphi
	:X\rightarrow \mathbb{R^{+}}$ be a lower semicontinuous function. Then $T$
	is called a generalized weakly contractive mapping if it satisfies the
	following condition: 
	\begin{equation*}
		\psi \left( s^{2}d\left( Tx,Ty\right) +\varphi \left( Tx\right) +\varphi
		\left( Ty\right) \right) \leq \psi \left( M\left( x,y,d,T,\varphi \right)
		\right) -\phi \left( M\left( x,y,d,T,\varphi \right) \right) ,
	\end{equation*}
	where 
	\begin{equation*}
		M(x,y,d,T,\varphi )=\max \{d(x,y)+\varphi (x)+\varphi (y),d(x,Tx)+\varphi
		(x)+\varphi (Tx),d(y,Ty)+\varphi (y)+\varphi (Ty)\}
	\end{equation*}
	for all $x,y\in X$, and $\psi $ is an altering distance function and $\phi : 
	\mathbb{R^{+}}\rightarrow \mathbb{R^{+}}$ is a lower semicontinuous function
	with $\phi (t)=0$ if and only if $t=0$.
\end{definition}

\begin{theorem}
	\label{thm3.2+}\cite{mdk} Let X be a complete $b$-rectangular metric space
	with parameter s $\geq 1$. If $T$ is a generalized weakly contractive
	mapping, then $T$ has a unique fixed point $z\in X$ such that $z=Tz$ and $%
	\varphi (z)=0$.
\end{theorem}

\begin{definition}
	\cite{mdk}Let $X$ be a complete $b$-rectangular metric space with metric $d$
	and parameter $s$, and $T:X\rightarrow X$ be a mapping. Also let $\varphi
	:X\rightarrow \mathbb{R^{+}}$ be a lower semicontinuous function. Then $T$
	is called a generalized $\left( \psi ,\varphi ,\phi \right) $ contractive
	mapping if it satisfies the following condition: 
	\begin{equation}
		\psi \left( s^{2}d\left( Tx,Ty\right) +\varphi \left( Tx\right) +\varphi
		\left( Ty\right) \right) \leq \phi \left( M\left( x,y,d,T,\varphi \right)
		\right) ,
	\end{equation}
	where 
	\begin{equation*}
		M(x,y,d,T,\varphi )=\max \{d(x,y)+\varphi (x)+\varphi (y),d(x,Tx)+\varphi
		(x)+\varphi (Tx),d(y,Ty)+\varphi (y)+\varphi (Ty)\}
	\end{equation*}
	for all $x,y\in X$, and $\psi :\mathbb{R^{+}}\rightarrow \mathbb{R^{+}}$ is
	an altering distance function and $\phi :\mathbb{R^{+}}\rightarrow \mathbb{\
		R^{+}}$ is a right upper semi-continuous function with the condition: $\psi
	(t)>\phi (t)$ for all $t>0$ and $\phi (t)=0$ if and only if $t=0$.
\end{definition}

\begin{theorem}
	\label{thm3.35++}\cite{mdk} Let X be a complete $b$-rectangular metric space
	with parameter s $\geq 1$ and $T:X\rightarrow X$ be a mapping. If $T$ is a
	generalized $\left( \psi ,\varphi ,\phi \right) $ contractive mapping then $%
	T $ has a unique fixed point $z\in X$ such that $z=Tz$ and $\varphi (z)=0$.
\end{theorem}

\section{Main results}

Aspired by idea of the generalized weakly contractive mapping on metric
space introduced by Cho \cite{CHO}, we introduce the notion of C-Class
function for weakly contractive mapping on rectangular $b$-metric space and
establish some fixed point on such mapping.

\begin{definition}
	\label{3.1} Let $X$ be a complete $b$-rectangular metric space with metric $%
	d $ and parameter $s$ and $T:X\rightarrow X$. Also let $\varphi
	:X\rightarrow R^{+}$ be a lower semicontinuous function. Then $T$ is called
	a $C $ class $-F-$generalized weakly contractive mapping if it satisfies the
	following condition: 
	\begin{equation}
		\psi \left( s^{2}d\left( Tx,Ty\right) +\varphi \left( Tx\right) +\varphi
		\left( Ty\right) \right) \leq F(\psi \left( M\left( x,y,d,T,\varphi \right)
		),\phi \left( M\left( x,y,d,T,\varphi \right) \right) \right) ,
	\end{equation}%
	where 
	\begin{equation*}
		M(x,y,d,T,\varphi )=\max \{d(x,y)+\varphi (x)+\varphi (y),d(x,Tx)+\varphi
		(x)+\varphi (Tx),d(y,Ty)+\varphi (y)+\varphi (Ty)\}
	\end{equation*}%
	for all $x,y\in X$, and $\psi \in \Phi ,\phi \in \Phi _{u}$, $\ F\in 
	\mathcal{C}$ .
\end{definition}

\begin{theorem}
	\label{thm3.2} Let X be a complete $b$-rectangular metric space with
	parameter s $\geq 1 $. If $T $ is a is called a $C $ class $-F-$generalized
	weakly contractive mapping, then $T $ has a unique fixed point $z\in X $
	such that $z = Tz $ and $\varphi(z) = 0 $.
\end{theorem}

\begin{proof}
	Let $x_{0}\in X$ be an arbitrary point in $X $. Then we define the sequence $%
	\left\lbrace x_{n}\right\rbrace$ by $x_{n+1} =Tx_{n}$, for all $n\in \mathbb{%
		\ N}.$
	
	If there exists $n_0\in \mathbb{N}$ such that $x_{n_0}=x_{n_0+1} =0$, then $%
	x_{n_0} $ is a fixed point of $T $.
	
	Next, we assume that $x_{n}\neq x_{n+1}$.\newline
	We claim that 
	\begin{equation*}
		\lim_{n\rightarrow +\infty }d\left( x_{n,}x_{n+1}\right) =0
	\end{equation*}
	and 
	\begin{equation*}
		\lim_{n\rightarrow +\infty }d\left( x_{n,}x_{n+2}\right) =0.
	\end{equation*}
	Letting $x=x_{n-1}$ and $y=x_{n}$ in (\ref{3.1}) for all $n\in $ $\mathbb{N}$
	, we have 
	\begin{eqnarray}
		&&\psi \left( s^{2}d\left( Tx_{n-1},Tx_{n}\right) +\varphi \left(
		Tx_{n-1}\right) +\varphi \left( Tx_{n}\right) \right)  \label{3.4} \\
		&&\qquad \leq F(\psi \left( M\left( x_{n-1},x_{n},d,T,\varphi \right)
		\right) ,\phi \left( M\left( x_{n-1},x_{n},d,T,\varphi \right) \right) ), 
		\notag
	\end{eqnarray}
	where 
	\begin{multline*}
		M\left( x_{n-1},x_{n},d,T,\varphi \right) =\max \{d(x_{n-1},x_{n})+\varphi
		(x_{n-1})+\varphi (x_{n}),d(x_{n-1},x_{n})+\varphi (x_{n-1}) \\
		+\varphi (x_{n}),d(x_{n},Tx_{n})+\varphi (x_{n})+\varphi (Tx_{n})\} \\
		=\max \{d(x_{n-1},x_{n})+\varphi (x_{n-1})+\varphi
		(x_{n}),d(x_{n},x_{n+1})+\varphi (x_{n})+\varphi (Tx_{n+1})\}.
	\end{multline*}
	If $M\left( x_{n-1},x_{n},d,T,\varphi \right) =d(x_{n},x_{n+1})+\varphi
	(x_{n})+\varphi (x_{n+1})$, then we have 
	\begin{align*}
		\psi \left( d\left( Tx_{n-1},Tx_{n}\right) +\varphi \left( Tx_{n-1}\right)
		+\varphi \left( x_{n+1}\right) \right) & =\psi \left( d\left(
		x_{n},x_{n+1}\right) +\varphi \left( x_{n}\right) +\varphi \left(
		x_{n+1}\right) \right) \\
		& \leq \psi \left( s^{2}d\left( x_{n},x_{n+1}\right) +\varphi \left(
		x_{n}\right) +\varphi \left( x_{n+1}\right) \right) \\
		& \leq F(\psi \left( d\left( x_{n},x_{n+1}\right) +\varphi \left(
		x_{n}\right) +\varphi \left( x_{n+1}\right) \right) \\
		& ,\phi \left( d\left( x_{n},x_{n+1}\right) +\varphi \left( x_{n}\right)
		+\varphi \left( x_{n+1}\right) \right) ),
	\end{align*}
	which implies 
	\begin{equation*}
		\phi \left( d\left( x_{n},x_{n+1}\right) +\varphi \left( x_{n}\right)
		+\varphi \left( x_{n+1}\right) \right) =0,
	\end{equation*}
	and so 
	\begin{equation*}
		d\left( x_{n},x_{n+1}\right) +\varphi \left( x_{n}\right) +\varphi \left(
		x_{n+1}\right) =0.
	\end{equation*}
	Hence 
	\begin{equation*}
		d\left( x_{n},x_{n+1}\right) =0\ and\ \varphi \left( x_{n}\right) =\varphi
		\left( x_{n+1}\right) =0,
	\end{equation*}
	which is a contradiction. Thus we have 
	\begin{equation}
		d\left( x_{n},x_{n+1}\right) +\varphi \left( x_{n}\right) +\varphi \left(
		x_{n+1}\right) \leq d\left( x_{n-1},x_{n}\right) +\varphi \left(
		x_{n-1}\right) +\varphi \left( x_{n}\right) ,\ for\ all\ n=1,2,3,...,
		\label{3.5}
	\end{equation}
	and 
	\begin{equation}
		M\left( x_{n-1},x_{n},d,T,\varphi \right) =d(x_{n-1},x_{n})+\varphi
		(x_{n-1})+\varphi (x_{n}),\ for\ all\ n=1,2,3,....  \label{3.6}
	\end{equation}
	for all $n=1,2,3,....$ It follows from (\ref{3.4}) that 
	\begin{equation*}
		\psi \left( d\left( x_{n},x_{n+1}\right) +\varphi \left( x_{n}\right)
		+\varphi \left( x_{n+1}\right) \right) \leq F(\psi \left( d\left(
		x_{n-1},x_{n}\right) +\varphi \left( x_{n-1}\right) +\varphi \left(
		x_{n}\right) \right)
	\end{equation*}
	\begin{equation}
		,\phi \left( d\left( x_{n-1},x_{n}\right) +\varphi \left( x_{n-1}\right)
		+\varphi \left( x_{n}\right) \right) ).  \label{3.7}
	\end{equation}
	It follows from (\ref{3.5}) that the sequence $\{d\left(
	x_{n},x_{n+1}\right) +\varphi \left( x_{n}\right) +\varphi \left(
	x_{n+1}\right) \}_{n\in \mathbb{N}}$ is nonincreasing.\newline
	Hence $d\left( x_{n},x_{n+1}\right) +\varphi \left( x_{n}\right) +\varphi
	\left( x_{n+1}\right) \rightarrow r$ as $n\rightarrow +\infty $ for some $%
	r\geq 0$. Assume $r>0$ and letting $n\rightarrow +\infty $ in (\ref{3.7})
	and using the continuity of $\psi $ and the lower semicontinuity of $\phi $,
	we have 
	\begin{align*}
		\psi \left( r\right) & \leq \psi \left( s^{2}r\right) \leq F(\psi \left(
		r\right) ,\liminf_{n\rightarrow \infty }\phi \left( d\left(
		x_{n},x_{n+1}\right) +\varphi \left( x_{n}\right) +\varphi \left(
		x_{n+1}\right) \right) ) \\
		& \leq F(\psi \left( r\right) ,\phi \left( r\right) ).
	\end{align*}
	It follows that $\psi \left( r\right) =0,$ $\ $or$\ \ \ ,\phi \left(
	r\right) =0$, hence we have $r=0$ and consequently, $\lim_{n\rightarrow
		+\infty }d\left( x_{n},x_{n+1}\right) +\varphi \left( x_{n}\right) +\varphi
	\left( x_{n+1}\right) =0.$ So 
	\begin{equation}
		\lim_{n\rightarrow +\infty }d\left( x_{n},x_{n+1}\right) =0,  \label{3.8}
	\end{equation}
	\begin{equation}
		\lim_{n\rightarrow +\infty }\varphi \left( x_{n}\right) =\lim_{n\rightarrow
			+\infty }\varphi \left( x_{n+1}\right) =0.  \label{3.9}
	\end{equation}
	
	Now, we shall prove that $T $ has a periodic point. Suppose that it is not
	the case. Then $x_{n}\neq x_{m}$ for all $n,m\in $ $\mathbb{N},\ n\neq m.$
	
	In (\ref{3.1}), letting $x=x_{n-1}$ and $y=x_{n+1}$, we have 
	\begin{eqnarray*}
		&&\psi \left( s^{2}d\left( Tx_{n-1},Tx_{n+1}\right) +\varphi \left(
		Tx_{n-1}\right) +\varphi \left( Tx_{n+1}\right) \right) \\
		&&\qquad \leq F(\psi \left( M\left( x_{n-1},x_{n+1},d,T,\varphi \right)
		\right) ,\phi \left( M\left( x_{n-1},x_{n+1},d,T,\varphi \right) \right) ),
	\end{eqnarray*}
	where 
	\begin{eqnarray*}
		M\left( x_{n-1},x_{n+1},d,T,\varphi \right) &=&\max
		\{d(x_{n-1},x_{n+1})+\varphi (x_{n-1})+\varphi (x_{n+1}), \\
		&&d(x_{n-1},x_{n})+\varphi (x_{n-1})+\varphi
		(x_{n}),d(x_{n+1},x_{n+2})+\varphi (x_{n+1})+\varphi (x_{n+2})\} \\
		&=&\max \{d(x_{n-1},x_{n+1})+\varphi (x_{n-1})+\varphi
		(x_{n+1}),d(x_{n-1},x_{n})+\varphi (x_{n-1})+\varphi (x_{n})\}.
	\end{eqnarray*}
	So we get 
	\begin{eqnarray}
		&&\psi \left( d\left( x_{n},x_{n+2}\right) +\varphi \left( x_{n}\right)
		+\varphi \left( x_{n+2}\right) \right) \leq \psi \left( s^{2}d\left(
		x_{n},x_{n+2}\right) +\varphi \left( x_{n}\right) +\varphi \left(
		x_{n+2}\right) \right)  \label{3.11} \\
		&\leq &F(\psi \left( \max \{d(x_{n-1},x_{n+1})+\varphi (x_{n-1})+\varphi
		(x_{n+1}),d(x_{n-1},x_{n})+\varphi (x_{n-1})+\varphi (x_{n})\}\right)  \notag
		\\
		&&,\phi \left( \max \{d(x_{n-1},x_{n+1})+\varphi (x_{n-1})+\varphi
		(x_{n+1}),d(x_{n-1},x_{n})+\varphi (x_{n-1})+\varphi (x_{n})\}\right) ). 
		\notag
	\end{eqnarray}
	Take $a_{n}=d\left( x_{n},x_{n+2}\right) +\varphi \left( x_{n}\right)
	+\varphi \left( x_{n+2}\right) $ and $b_{n}=d(x_{n},x_{n+1})+\varphi
	(x_{n})+\varphi (x_{n+1}).$\newline
	Then by (\ref{3.11}), one can write 
	\begin{align*}
		\psi \left( a_{n}\right) & \leq F(\psi \left( \max \left(
		a_{n-1},b_{n-1}\right) \right) ,\phi \left( \max \left(
		a_{n-1},b_{n-1}\right) \right) ) \\
		& \leq \psi \left( \max \left( a_{n-1},b_{n-1}\right) \right) .
	\end{align*}
	Since $\psi $ is increasing, we get 
	\begin{equation*}
		a_{n}\leq \max \left\{ a_{n-1},b_{n-1}\right\} . \\
	\end{equation*}
	By (\ref{3.5}), we have 
	\begin{equation*}
		b_{n}\leq b_{n-1}\leq \max \left\{ a_{n-1},b_{n-1}\right\} ,
	\end{equation*}
	which implies that 
	\begin{equation*}
		\max \left\{ a_{n},b_{n}\right\} \leq \max \left\{ a_{n-1},b_{n-1}\right\} , 
		\text{ }\forall n\in \mathbb{N}. \\
	\end{equation*}
	Therefore, the sequence $\max \left\{ a_{n-1},b_{n-1}\right\} _{n\in \mathbb{%
			\ N}}$ is a nonnegative decreasing sequence of real numbers. Thus there
	exists $\lambda \geq 0$ such that 
	\begin{equation*}
		\lim_{n\rightarrow +\infty }\max \left\{ a_{n},b_{n}\right\} =\lambda .
	\end{equation*}
	Assume that $\lambda >0$. By (\ref{3.8}), it is obvious that 
	\begin{equation}
		\lambda =\lim_{n\rightarrow +\infty }\sup a_{n}=\lim_{n\rightarrow +\infty
		}\sup \max \left\{ a_{n},b_{n}\right\} =\lim_{n\rightarrow +\infty }\max
		\left\{ a_{n},b_{n}\right\} .  \label{3.12}
	\end{equation}
	Taking $\limsup_{n}\rightarrow +\infty $ in (\ref{3.11}), using (\ref{3.12})
	and using the properties of $\psi $ and $\phi $, we obtain 
	\begin{align*}
		\psi (\lambda )& =\psi \left( \limsup_{n\rightarrow +\infty }a_{n}\right) \\
		& =\limsup_{n\rightarrow +\infty }\psi \left( a_{n}\right) \\
		& \leq \limsup_{n\rightarrow +\infty }\psi \left( \max \left\{
		a_{n},b_{n}\right\} \right) -\liminf_{n\rightarrow +\infty }\phi \left( \max
		\left\{ a_{n},b_{n}\right\} \right) \\
		& \leq F(\psi \left( \lim_{n\rightarrow +\infty }\max \left\{
		a_{n},b_{n}\right\} \right) ,\phi \left( \lim_{n\rightarrow +\infty }\max
		\left\{ a_{n},b_{n}\right\} \right) ) \\
		& =F(\psi (\lambda ),\phi (\lambda )),
	\end{align*}
	which implies that $\psi \left( r\right) =0,$ $\ $or$\ \ \ ,\phi \left(
	r\right) =0$, a contradiction. Thus, from (\ref{3.12}), 
	\begin{equation*}
		\limsup_{n\rightarrow +\infty }a_{n}=0
	\end{equation*}
	and hence 
	\begin{equation*}
		\lim_{n\rightarrow +\infty }d\left( x_{n,}x_{n+2}\right) =0.
	\end{equation*}
	
	Next, we shall prove that $\left\lbrace x_{n}\right\rbrace _{n\in \mathbb{N}%
	} $ is a Cauchy sequence, i.e, $\lim_{n,m\rightarrow +\infty }d\left(
	x_{n,}x_{m}\right) =0$ for all $n,m\in \mathbb{N}$. Suppose to the contrary.
	By Lemma $\ref{2?4}$, there is an $\varepsilon $ $>0$ such that for an
	integer $k $ there exist two sequences $\left\lbrace n_{\left( k\right)
	}\right\rbrace $ and $\left\lbrace m_{\left( k\right) }\right\rbrace $ such
	that
	
	\begin{itemize}
		\item[i)] $\varepsilon \leq \lim_{k\rightarrow +\infty }\inf d\left(
		x_{m_{\left( k\right) }},x_{n_{\left( k\right)}}\right) \leq
		\lim_{k\rightarrow +\infty }\sup d\left( x_{m_{\left( k\right)
		}},x_{n_{\left( k\right)}}\right)\leq s\varepsilon ,$
		
		\item[ii)] $\varepsilon \leq \lim_{k\rightarrow +\infty }\inf d\left(
		x_{n_{\left( k\right) }},x_{m_{\left( k\right)}+1}\right) \leq
		\lim_{k\rightarrow +\infty }\sup d\left( x_{n_{\left( k\right)
		}},x_{m_{\left( k\right)}+1}\right)\leq s\varepsilon ,$
		
		\item[iii)] $\varepsilon \leq \lim_{k\rightarrow +\infty }\inf d\left(
		x_{m_{\left( k\right) }},x_{n_{\left( k\right)}+1}\right) \leq
		\lim_{k\rightarrow +\infty }\sup d\left( x_{m_{\left( k\right)
		}},x_{n_{\left( k\right)}+1}\right)\leq s\varepsilon ,$
		
		\item[vi)] $\frac{\varepsilon}{s} \leq \lim_{k\rightarrow +\infty }\inf
		d\left( x_{m_{\left( k\right) +1}},x_{n_{\left( k\right)}+1}\right) \leq
		\lim_{k\rightarrow +\infty }\sup d\left( x_{m_{\left(
				k\right)}+1},x_{n_{\left( k\right)}+1}\right)\leq s^{2}\varepsilon .$
	\end{itemize}
	
	From (\ref{3.1}) and by setting $x=x_{m_{\left( k\right) }}$ and $%
	y=x_{n_{\left( k\right) }}$, we have 
	\begin{eqnarray*}
		&&M(x_{m_{(k)}},x_{n_{(k)}},d,T,\varphi )=\max
		\{d(x_{m_{(k)}},x_{n_{(k)}})+\varphi (x_{m_{(k)}})+\varphi (x_{m_{(k)}}), \\
		&&\quad d(x_{m_{(k)}},x_{m_{(k)}+1})+\varphi (x_{m_{(k)}})+\varphi
		(x_{m_{(k)}+1}),d(x_{n_{(k)}},x_{n_{(k)}+1})+\varphi (x_{n_{(k)}})+\varphi
		(x_{n_{(k)}+1})\}.
	\end{eqnarray*}%
	Taking the limit as $k\rightarrow +\infty $ and using (\ref{3.8}), (\ref{3.9}
	) and $(iii)$ of Lemma \ref{2?4}, we have 
	\begin{equation}
		\lim_{k\rightarrow +\infty }M\left( x_{m_{\left( k\right) }},x_{n_{\left(
				k\right) }},d,T,\varphi \right) \leq s\varepsilon .  \label{3.15}
	\end{equation}%
	Now letting $x=x_{m_{\left( k\right) }}$ and $y=x_{n_{\left( k\right) }}$ in
	(\ref{3.1}), we have 
	\begin{multline*}
		\psi \left[ s^{2}d\left( x_{m_{\left( k\right) }+1},x_{n_{\left( k\right)
			}+1}\right) +\varphi \left( m_{\left( k\right) }+1\right) +\varphi \left(
		n_{\left( k\right) }+1\right) \right] \\
		\leq F(\psi \left[ d\left( x_{m_{\left( k\right) }+1},x_{n_{\left( k\right)
			}+1}\right) +\varphi \left( m_{\left( k\right) }+1\right) +\varphi \left(
		n_{\left( k\right) }+1\right) \right] \\
		,\phi \left[ d\left( x_{m_{\left( k\right) }},x_{n_{\left( k\right)
			}+1}\right) +\varphi \left( m_{\left( k\right) }\right) +\varphi \left(
		n_{\left( k\right) }\right) \right] ).
	\end{multline*}%
	Letting $k\rightarrow +\infty $, using (\ref{3.8}), (\ref{3.9}), (\ref{3.15}
	), and applying the continuity of $\psi $ and the lower semicontinuity of $%
	\phi $, we have 
	\begin{equation*}
		\lim_{k\rightarrow +\infty }\psi \left[ s^{2}d\left( x_{m_{\left( k\right)
			}+1},x_{n_{\left( k\right) }+1}\right) \right] \leq F(\psi (s\varepsilon
		),\phi (s\varepsilon )).
	\end{equation*}%
	Using (\ref{3.15}) and $(iv)$ of Lemma $\ref{2?4}$, we obtain 
	\begin{equation*}
		\psi (s\varepsilon )=\psi \left( s^{2}\frac{\varepsilon }{s}\right) \leq
		\lim_{k\rightarrow +\infty }\sup \psi \left[ s^{2}d\left( x_{m_{\left(
				k\right) }+1},x_{n_{\left( k\right) }+1}\right) \right] \leq F(\psi
		(s\varepsilon ),\phi (s\varepsilon )).
	\end{equation*}%
	This is a contradiction. Thus 
	\begin{equation*}
		\lim_{n,m\rightarrow +\infty }d\left( x_{m},x_{n}\right) =0.
	\end{equation*}%
	Hence $\left\{ x_{n}\right\} $ is a Cauchy sequence in $X$. By completeness
	of $\left( X,d\right) $, there exists $z\in X$ such that 
	\begin{equation*}
		\lim_{n\rightarrow +\infty }d\left( x_{n},z\right) =0.
	\end{equation*}%
	Since $\varphi $ is lower semicontinuous, we get 
	\begin{equation*}
		\varphi \left( z\right) \leq \liminf_{n\rightarrow +\infty }\varphi \left(
		x_{n}\right) \leq \lim_{n\rightarrow +\infty }\varphi \left( x_{n}\right) =0,
		\\
	\end{equation*}%
	which implies 
	\begin{eqnarray}
		\varphi \left( z\right) &=&0.  \label{3.17} \\
		&&  \notag
	\end{eqnarray}%
	Now, putting $x=x_{n}$ and $y=z$ in (\ref{3.1}), we have 
	\begin{eqnarray*}
		M\left( x_{n},z,d,T,\varphi \right) &=&\max \{d\left( x_{n},z\right)
		+\varphi \left( x_{n}\right) +\varphi \left( z\right) , \\
		&&d\left( x_{n},x_{n+1}\right) +\varphi \left( x_{n}\right) +\varphi \left(
		x_{n+1}\right) ,d\left( z,Tz\right) +\varphi \left( z\right) +\varphi \left(
		Tz\right) \}.
	\end{eqnarray*}%
	Taking the limit as $n\rightarrow +\infty $ and using (\ref{3.8}), (\ref{3.9}
	) and (\ref{3.17}), we have 
	\begin{equation*}
		\lim_{n\rightarrow +\infty }M\left( x_{n},z,d,T,\varphi \right) =d\left(
		z,Tz\right) +\varphi \left( Tz\right) .
	\end{equation*}%
	Since $x_{n}\rightarrow z$ as $n\rightarrow +\infty $, from Lemma \ref{2.3}
	, we conclude that 
	\begin{equation*}
		\frac{1}{s}d\left( z,Tz\right) \leq \lim_{n\rightarrow +\infty }\sup d\left(
		Tx_{n},Tz\right) \leq sd\left( z,Tz\right) .
	\end{equation*}%
	Hence 
	\begin{equation*}
		sd\left( z,Tz\right) =s^{2}\frac{1}{s}d\left( z,Tz\right) \leq
		\lim_{n\rightarrow +\infty }\sup s^{2}d\left( Tx_{n},Tz\right) ,
	\end{equation*}%
	which implies 
	\begin{equation*}
		\lim_{n\rightarrow +\infty }\sup \left[ sd\left( z,Tz\right) +\varphi \left(
		x_{n+1}\right) +\varphi \left( Tz\right) \right] \leq \lim_{n\rightarrow
			+\infty }\sup \left[ s^{2}d\left( Tx_{n},Tz\right) +\left( x_{n+1}\right)
		+\varphi \left( Tz\right) \right] .
	\end{equation*}%
	Then using (\ref{3.1}), we have 
	\begin{align*}
		\psi \left[ s^{2}d\left( Tx_{n},Tz\right) +\varphi \left( Tx_{n}\right)
		+\varphi \left( Tz\right) \right] & =\psi \left[ s^{2}d\left(
		x_{n+1},Tz\right) +\varphi \left( x_{n+1}\right) +\varphi \left( Tz\right) %
		\right] \\
		& \leq F(\psi \left[ M\left( x_{n},z,d,T,\varphi \right) \right] ,\phi \left[
		M\left( x_{n},z,d,T,\varphi \right) \right] ).
	\end{align*}%
	Letting $n\rightarrow +\infty $ and using the continuity of $\psi $ and the
	lower semicontinuity of $\phi $, we have 
	\begin{multline*}
		\psi \left[ \lim_{n\rightarrow +\infty }\sup \left( sd\left( z,Tz\right)
		+\varphi \left( x_{n+1}\right) +\varphi \left( Tz\right) \right) \right] \\
		\leq \psi \left[ \lim_{n\rightarrow +\infty }\sup \left( s^{2}d\left(
		Tx_{n},Tz\right) +\left( x_{n+1}\right) +\varphi \left( Tz\right) \right) %
		\right] \\
		\leq F(\psi \left[ \lim_{n\rightarrow +\infty }\sup M\left(
		x_{n},z,d,T,\varphi \right) \right] ,\lim_{n\rightarrow +\infty }\phi \left[
		M\left( x_{n},z,d,T,\varphi \right) \right] ),
	\end{multline*}%
	which implies 
	\begin{equation*}
		\psi \left[ sd\left( z,Tz\right) +\varphi \left( Tz\right) \right] \leq
		F(\psi \left[ d\left( z,Tz\right) +\varphi \left( Tz\right) \right] ,\phi %
		\left[ d\left( z,Tz\right) +\varphi \left( Tz\right) \right] ).
	\end{equation*}%
	This holds if and only if \ $\psi \left( d\left( z,Tz\right) +\varphi \left(
	Tz\right) \right) =0$\ \ or $\phi \left( d\left( z,Tz\right) +\varphi \left(
	Tz\right) \right) =0$ and from the property of $\psi ,\phi $, we have 
	\begin{equation*}
		d\left( z,Tz\right) +\varphi \left( Tz\right) =0.
	\end{equation*}%
	Hence $d\left( z,Tz\right) =0$ and so $z=Tz$ and $\varphi \left( Tz\right)
	=0 $. It is a contradiction to the assumption: that $T$ does not have a
	periodic point. Thus $T$ has a periodic point, say, $z$ of period $n$.
	Suppose that the set of fixed points of $T$ is empty. Then we have 
	\begin{equation*}
		q>0\ and\ d(z,Tz)>0.
	\end{equation*}%
	Since $T$ has a periodic point, $z=T^{n}z$. Letting $x=T^{n-1}z$ and $%
	y=T^{n}z$, we obtain 
	\begin{multline*}
		M\left( T^{n}z,T^{n-1}z,d,T,\varphi \right) =\max \{d\left(
		T^{n-1}z,T^{n}z\right) +\varphi \left( T^{n-1}z\right) +\varphi \left(
		T^{n}z\right) , \\
		d\left( T^{n-1}z,T^{n}z\right) +\varphi \left( T^{n-1}z\right) +\varphi
		\left( T^{n}z\right) ,d\left( T^{n}z,TT^{n}z\right) +\varphi \left(
		T^{n}z\right) +\varphi \left( TT^{n}z\right) \}.
	\end{multline*}
	
	By a similar method to (\ref{3.6}), we conclude that 
	\begin{equation*}
		M\left( T^{n}z,T^{n-1}z,d,T,\varphi \right) =d\left( T^{n-1}z,T^{n}z\right)
		+\varphi \left( T^{n-1}z\right) +\varphi \left( T^{n}z\right) .
	\end{equation*}
	From (\ref{3.1}), we have 
	\begin{align*}
		\psi \left[ s^{2}d\left( z,Tz\right) +\varphi \left( T^{n}z\right) +\varphi
		\left( T^{n+1}z\right) \right] & =\psi \left[ s^{2}d\left(
		T^{n}z,T^{n+1}z\right) +\varphi \left( T^{n}z\right) +\varphi \left(
		T^{n+1}z\right) \right] \\
		& \leq F(\psi \left[ d\left( T^{n-1}z,T^{n}z\right) +\varphi \left(
		T^{n-1}z\right) +\varphi \left( T^{n}z\right) \right] \\
		& ,\phi \left[ d\left( T^{n-1}z,T^{n}z\right) +\varphi \left(
		T^{n-1}z\right) +\varphi \left( T^{n}z\right) \right] ) \\
		& \leq \psi \left[ s^{2}d\left( T^{n-1}z,T^{n}z\right) +\varphi \left(
		T^{n-1}z\right) +\varphi \left( T^{n}z\right) \right] \\
		& \vdots \\
		& \leq F(\psi \left[ d\left( z,Tz\right) +\varphi \left( z\right) +\varphi
		\left( Tz\right) \right] \\
		& ,\phi \left[ d\left( z,Tz\right) +\varphi \left( z\right) +\varphi \left(
		Tz\right) \right] )
	\end{align*}
	Taking the limit as $n\rightarrow +\infty $ and applying the continuity of $%
	\psi $ and the lower semicontinuity of $\phi $, we have 
	\begin{equation*}
		\psi \left[ s^{2}d\left( z,Tz\right) \right] \leq F(\psi \left[ d\left(
		z,Tz\right) \right] ,\phi \left[ d\left( z,Tz\right) \right] ).
	\end{equation*}
	Hence $d(z,Tz)=0$, which is a contradiction. Thus the set of fixed points of
	T is non-empty, that is, $T$ has at least one fixed point.
	
	Suppose that $z,u\in X$ are two fixed points of $T$ such that $u\neq z$.
	Then $Tz=z $ and $Tu=u $.
	
	Letting $x=z$ and $y=u$ in (\ref{3.1}), we have 
	\begin{equation*}
		\psi \left( s^{2}d\left( Tz,Tu\right) +\varphi \left( Tz\right) +\varphi
		\left( Tu\right) \right) =\psi \left( s^{2}d\left( z,u\right) \right) \leq
		F(\psi \left( M\left( z,u,d,T,\varphi \right) \right) ,\phi \left( M\left(
		z,u,d,T,\varphi \right) \right) ),
	\end{equation*}
	where 
	\begin{align*}
		M(z,u,d,T,\varphi )& =\max \{d(z,u)+\varphi (z)+\varphi (u),d(z,Tz)+\varphi
		(z)+\varphi (Tz),d(u,Tu)+\varphi (u)+\varphi (Tu)\} \\
		& =d(z,u).
	\end{align*}
	So 
	\begin{equation*}
		\psi \left( s^{2}d\left( z,u\right) \right) \leq F(\psi \left( d\left(
		z,u\right) \right) ,\phi \left( d\left( z,u\right) \right) ).
	\end{equation*}
	This holds if $\phi \left( d\left( z,u\right) \right) =0$ and so we have $%
	(d\left( z,u\right) =0$. Hence $z=u$ and $T$ has a unique fixed point.
\end{proof}

\begin{corollary}
	\label{cor3.3} Let $(X,d)$ be a complete $b$-rectangular metric space and $%
	T:X\rightarrow X$ be a mapping. Suppose that there exists $k \in \left] 0,1 %
	\right[ $ such that for all $x,y\in X$, 
	\begin{multline*}
		s^{2}d\left( Tx,Ty\right)+ \varphi\left( Tx\right) + \varphi\left(
		Ty\right)\leq \\
		k \max\lbrace d\left( x,y\right)+\varphi\left( x\right)+\varphi\left(
		y\right),d\left( x,Tx\right)+\varphi\left( x\right)+ \varphi\left(
		Tx\right),d\left( y,Ty\right)+\varphi\left( y\right)+\varphi\left(
		Ty\right)\rbrace,
	\end{multline*}
	where $\varphi : \mathbb{R^{+}}\rightarrow \mathbb{R^{+}}$ is a lower
	semicontinuous function. Then $T$ has a unique fixed point.
\end{corollary}

\begin{proof}
	It suffices to take Take $\psi (t)=t$ and $F(s,t)=ks$ in Theorem $\ref%
	{thm3.2}$.
\end{proof}

\begin{corollary}
	Let $(X,d)$ be a complete $b$-rectangular metric space and $T:X\rightarrow X$
	be a mapping. Suppose that there exists $\alpha \in \left] 0,\frac{1}{2} %
	\right[ $ such that for all $x,y\in X$, 
	\begin{equation}  \label{3.24}
		s^{2}d\left( Tx,Ty\right)+ \varphi\left( Tx\right) + \varphi\left( Ty\right)
		\leq\alpha \left[\left( d\left( Tx,x\right)+\varphi\left( x\right)+
		\varphi\left( Tx\right) +d+\varphi\left( y\right)+\varphi\left(
		Ty\right)+\left(Ty,y\right)\right) \right],
	\end{equation}
	where $\varphi : \mathbb{R^{+}}\rightarrow \mathbb{R^{+}}$ is a lower
	semicontinuous function. Then $T$ has a unique fixed point.
\end{corollary}

\begin{proof}
	Let $k=2\alpha$. Then $k\in\left] 0,1\right[ $. Also, if (\ref{3.24}) holds,
	then 
	\begin{multline*}
		s^{2}d\left( Tx,Ty\right)+ \varphi\left( Tx\right) + \varphi\left( Ty\right)
		\leq\alpha \left[ d\left( Tx,x\right)+\varphi\left( x\right)+ \varphi\left(
		Tx\right) +d+\varphi\left( y\right)+\varphi\left(
		Ty\right)+\left(Ty,y\right) \right] \\
		=k \frac{ \left[ d\left( Tx,x\right)+\varphi\left( x\right)+ \varphi\left(
			Tx\right) +d+\varphi\left( y\right)+\varphi\left(
			Ty\right)+\left(Ty,y\right) \right]}{2} \\
		\leq k \max\lbrace d\left( x,Tx\right)+\varphi\left( x\right)+ \varphi\left(
		Tx\right),d\left( y,Ty\right)+\varphi\left( y\right)+\varphi\left(
		Ty\right)\rbrace \\
		\leq k \max\lbrace d\left( x,y\right)+\varphi\left( x\right)+\varphi\left(
		y\right),d\left( x,Tx\right)+\varphi\left( x\right)+ \varphi\left(
		Tx\right),d\left( y,Ty\right)+\varphi\left( y\right)+\varphi\left(
		Ty\right)\rbrace.
	\end{multline*}
	Thus it suffices to apply Corollary $\ref{cor3.3}$.
\end{proof}

\begin{corollary}
	Let $(X,d)$ be a complete $b$-rectangular metric space and $T:X\rightarrow X$
	be a mapping. Suppose that there exists $\lambda \in \left] 0,\frac{1}{3} %
	\right[ $ such that for all $x,y\in X$, 
	\begin{multline}  \label{3.25}
		s^{2}d\left( Tx,Ty\right)+ \varphi\left( Tx\right) + \varphi\left( Ty\right)
		\leq \\
		\lambda\left[d\left( x,y\right)+\varphi\left( x\right)+ \varphi\left(
		y\right)+ d\left( Tx,x\right)+\varphi\left( x\right)+ \varphi\left(
		Tx\right) +\varphi\left( y\right)+\varphi\left( Ty\right)+d\left(Ty,y\right) %
		\right],
	\end{multline}
	where $\varphi : \mathbb{R^{+}}\rightarrow \mathbb{R^{+}}$ is a lower
	semicontinuous function. Then $T$ has a unique fixed point.
\end{corollary}

\begin{proof}
	Let $k=3\lambda$. Then $k\in\left] 0,1\right[ $. Also, if (\ref{3.25})
	holds, then 
	\begin{multline*}
		s^{2}d\left( Tx,Ty\right)+ \varphi\left( Tx\right) + \varphi\left( Ty\right)
		\\
		\leq\lambda \left[d\left( x,y\right)+\varphi\left( x\right)+ \varphi\left(
		y\right)+ d\left( Tx,x\right)+\varphi\left( x\right)+ \varphi\left(
		Tx\right) +\varphi\left( y\right)+\varphi\left( Ty\right)+d\left(Ty,y\right) %
		\right] \\
		=k \frac{ \left[ d\left( x,y\right)+\varphi\left( x\right)+ \varphi\left(
			y\right)+d\left( Tx,x\right)+\varphi\left( x\right)+ \varphi\left( Tx\right)
			+\varphi\left( y\right)+\varphi\left( Ty\right)+d\left(Ty,y\right) \right]}{%
			3 } \\
		\leq k \max\lbrace d\left( x,y\right)+\varphi\left( x\right)+\varphi\left(
		y\right),d\left( x,Tx\right)+\varphi\left( x\right)+ \varphi\left(
		Tx\right),d\left( y,Ty\right)+\varphi\left( y\right)+\varphi\left(
		Ty\right)\rbrace.
	\end{multline*}
	Thus it suffices to apply Corollary \ref{cor3.3}.
\end{proof}

\begin{corollary}
	Let $d\left( X,d\right) $ be a complete $b$-rectangular metric space with
	parameter $s>1$ and $T$ be a self mapping on $X$. If there exists $k \in %
	\left] 0,1\right[ $ such that for all $x,y\in X$, 
	\begin{multline*}
		s^{2} d\left( Tx,Ty\right)+\varphi\left( Tx\right)+\varphi\left( Tx\right)
		\leq \\
		k\left[ \beta_{1}\left( d\left( x,y\right)+\varphi\left(
		x\right)+\varphi\left( y\right) \right) +\beta _{2}\left( d\left(
		Tx,x\right)+\varphi\left( x\right)+\varphi\left( Tx\right) \right) +\beta
		_{3}\left( d\left(Ty,y\right)+\varphi\left( y\right)+\varphi\left(
		Ty\right)\right)\right],
	\end{multline*}
	where $\beta _{i}\geq0$ for $i\in\{1,2,3\},$ $\sum\limits_{\substack{ i=0 }}
	^{i=3} {\beta_i}\leq1$, $\varphi $ is a lower semicontinuous function. Then $%
	T$ has a unique fixed point.
\end{corollary}

\begin{proof}
	Take $\psi (t)=t$ and $F(s,t)=ks$. Then it suffices to apply Corollary \ref%
	{cor3.3}.
\end{proof}

\begin{corollary}
	\label{cor2022+94+1} Let $(X,d)$ be a complete $b$-rectangular metric space
	and $T:X\rightarrow X$ be a mapping. Suppose that for all $x,y\in X$, 
	\begin{equation}
		\psi \left( s^{2}d\left( Tx,Ty\right) +\varphi \left( Tx\right) +\varphi
		\left( Ty\right) \right) \leq \psi (M\left( x,y,d,T,\varphi \right) )\log
		_{a+\phi \left( M\left( x,y,d,T,\varphi \right) \right) }a,  \label{2022+1}
	\end{equation}%
	where 
	\begin{equation*}
		M(x,y,d,T,\varphi )=\max \{d(x,y)+\varphi (x)+\varphi (y),d(x,Tx)+\varphi
		(x)+\varphi (Tx),d(y,Ty)+\varphi (y)+\varphi (Ty)\}
	\end{equation*}%
	and $\varphi :\mathbb{R^{+}}\rightarrow \mathbb{R^{+}}$ is a lower
	semicontinuous function also $\psi ,\varphi \in \Phi $, $a>1$\ . Then $T$
	has a unique fixed point.
\end{corollary}

\begin{proof}
	It suffices to take $F(s,t)=s\log _{t+a}a$, $a>1$, in Theorem $\ref{thm3.2}$.
\end{proof}

\begin{corollary}
	\label{cor2022+94+1} Let $(X,d)$ be a complete $b$-rectangular metric space
	and $T:X\rightarrow X$ be a mapping. Suppose that for all $x,y\in X$, 
	\begin{equation}
		\psi \left( s^{2}d\left( Tx,Ty\right) +\varphi \left( Tx\right) +\varphi
		\left( Ty\right) \right) \leq \frac{\psi (M\left( x,y,d,T,\varphi \right) )}{%
			1+\phi \left( M\left( x,y,d,T,\varphi \right) \right) },  \label{2022+1}
	\end{equation}%
	where 
	\begin{equation*}
		M(x,y,d,T,\varphi )=\max \{d(x,y)+\varphi (x)+\varphi (y),d(x,Tx)+\varphi
		(x)+\varphi (Tx),d(y,Ty)+\varphi (y)+\varphi (Ty)\}
	\end{equation*}%
	and $\varphi :\mathbb{R^{+}}\rightarrow \mathbb{R^{+}}$ is a lower
	semicontinuous function also $\psi \in \Phi ,\phi \in \Phi _{u}$, \ . Then $%
	T $ has a unique fixed point.
\end{corollary}

\begin{proof}
	It suffices to take $F(s,t)=\frac{s}{1+t}$, in Theorem $\ref{thm3.2}$.
\end{proof}

\begin{example}
	Let $X=A\cup B$, where $A=\left\{ 0,\frac{1}{5},\frac{1}{9},\frac{1}{16}%
	\right\} $ and $B=\left[ \frac{1}{2},1\right] $. Define $d:X\times
	X\rightarrow \left[ 0,+\infty \right[ $ as follows: 
	\begin{equation*}
		\left\{ \begin{aligned} d(x, y) &=d(y, x)\ for \ all \  x,y\in X;\\ d(x, y)
			&=0\Leftrightarrow y= x\\ \end{aligned}\right.
	\end{equation*}%
	and 
	\begin{equation*}
		\left\{ \begin{aligned} d\left( 0,\frac{1}{9}\right) =d\left(
			\frac{1}{5},\frac{1}{16}\right) =0,1\\ d\left(0,\frac{1}{5}\right) =d\left(
			\frac{1}{5},\frac{1}{9}\right) =0,5\\ d\left( 0,\frac{1}{16}\right) =d\left(
			\frac{1}{9},\frac{1}{16}\right)=0,05\\ d\left( x,y\right) =\left( \vert
			x-y\vert\right) ^{2} \ otherwise. \end{aligned}\right.
	\end{equation*}%
	Then $(X,d)$ is a $b$-rectangular metric space with coefficient $s=3$.
	However we have the following:
	
	\item[1)] $(X,d)$ is not a metric space, since $d\left( \frac{1}{5},\frac{1}{%
		9}\right) =0.5>0.15=d\left( \frac{1}{5},\frac{1}{16}\right) +d\left( \frac{1%
	}{16},\frac{1}{9}\right) $.
	
	\item[2)] $(X,d)$ is not a $b$-metric space for s=3, since $d\left( \frac{1}{%
		5},\frac{1}{9}\right) =0.5>0.45=3\left[ d\left( \frac{1}{5},\frac{1}{16}%
	\right) +d\left( \frac{1}{16},\frac{1}{9}\right) \right] $.
	
	\item[3)] $(X,d)$ is not a rectangular metric space, since $d\left( \frac{1}{%
		5},\frac{1}{9}\right) =0.5>0.25=d\left( \frac{1}{5},\frac{1}{16}\right)
	+d\left( \frac{1}{16},0\right) +d\left( 0,\frac{1}{9}\right) $.\newline
	Define a mapping $T:X\rightarrow X$ by 
	\begin{equation*}
		T(x)=\left\{ \begin{aligned} \frac{1}{16} & \ if \ x\in \left[\frac{1}{2},1
			\right]\\ 0 & \ if \ x\in A.\\ \end{aligned}\right.
	\end{equation*}%
	Then $T(x)\in X$ for all $x\in X$. Let 
	\begin{equation*}
		\varphi (t)=\left\{ \begin{aligned} t & \ if \ t\in \left[0,1 \right]\\ 2t &
			\ if \ t>1\\ \end{aligned}\right.
	\end{equation*}%
	\begin{equation*}
		\varphi (t)=\left\{ \begin{aligned} \frac{t}{16} & \ if \ t\in \left[0,1
			\right]\\ \frac{t}{8} & \ if \ t>1 \end{aligned}\right.
	\end{equation*}%
	and 
	\begin{equation*}
		\psi (t)=\frac{3t}{2}
	\end{equation*}%
	and 
	\begin{equation*}
		F(s,t)=s-t.
	\end{equation*}%
	Then $\psi $ is an altering distance function and $\varphi $ is a lower
	semicontinuous function and $\phi $ is a lower semicontinuous function such
	that $\psi (t)=0\Leftrightarrow t=0$, $\phi (t)=0\Leftrightarrow t=0$ and $%
	\varphi (t)=0\Leftrightarrow t=0$ and $F$ C-class function . \newline
	Consider the following possibilities:
	
	Case I: $x,y\in \left\{ 0,\frac{1}{5},\frac{1}{9},\frac{1}{16}\right\} .$
	
	Assume that $x\geq y.$ Then 
	\begin{equation*}
		\psi \left( s^{2}d(Tx,Ty)+\varphi (Tx)+\varphi (Ty)\right) =\psi \left(
		9.d(0,0)+\varphi (0)+\varphi ()\right) =\psi (0)=0.
	\end{equation*}%
	Also 
	\begin{equation*}
		d(x,y)+\varphi (x)+\varphi (y)=d(x,y)+x+y,
	\end{equation*}%
	\begin{equation*}
		d(x,Tx)+\varphi (x)+\varphi (Tx)=d(x,0)+x,
	\end{equation*}%
	\begin{equation*}
		d(y,Ty)+\varphi (y)+\varphi (Ty)=d(y,0)+y
	\end{equation*}%
	and 
	\begin{equation*}
		M(x,y,d,T,\varphi )=\max \{d(x,y)+x+y,d(x,0)+x,d(y,0)+y\}.
	\end{equation*}%
	Since $x\geq y$, we have 
	\begin{equation*}
		M(x,y,d,T,\varphi )=\max \{d(x,y)+x+y,d(x,0)+x\}.
	\end{equation*}%
	If 
	\begin{equation*}
		M(x,y,d,T,\varphi )=d(x,y)+x+y\geq \frac{1}{20},
	\end{equation*}%
	then 
	\begin{eqnarray*}
		\psi \left( M(x,y,d,T,\varphi )-\phi \left( M(x,y,d,T,\varphi )\right)
		\right) &=&\psi \left( d(x,y)+x+y)-\phi \left( d(x,y)+x+y\right) \right) \\
		&=&\frac{23}{16}\left( d(x,y)+x+y\right) \geq 0
	\end{eqnarray*}%
	and so 
	\begin{equation*}
		0=\psi \left( s^{2}d(Tx,Ty)+\varphi (Tx)+\varphi (Ty)\right) \leq \psi
		\left( M(x,y,d,T,\varphi )-\phi \left( M(x,y,d,T,\varphi )\right) \right) .
	\end{equation*}%
	If 
	\begin{equation*}
		M(x,y,d,T,\varphi )=d(x,0)+x\geq \frac{1}{20},
	\end{equation*}%
	then 
	\begin{eqnarray*}
		\psi \left( s^{2}d(Tx,Ty)+\varphi (Tx)+\varphi (Ty)\right) &\leq &\frac{3}{2}%
		\cdot \frac{1}{20}-\frac{1}{20}\cdot \frac{1}{16}=\frac{23}{320} \\
		&\leq &\psi \left( M(x,y,d,T,\varphi )-\phi \left( M(x,y,d,T,\varphi
		)\right) \right) .
	\end{eqnarray*}
	
	Assume that $x<y$. Then 
	\begin{equation*}
		M(x,y,d,T,\varphi )=\max \{d(x,y)+x+y,d(y,0)+y\}.
	\end{equation*}%
	If 
	\begin{equation*}
		M(x,y,d,T,\varphi )=d(x,y)+x+y\geq \frac{1}{20}+\frac{1}{16}=\frac{9}{80},
	\end{equation*}%
	then 
	\begin{eqnarray*}
		\psi \left( s^{2}d(Tx,Ty)+\varphi (Tx)+\varphi (Ty)\right) &\leq &\frac{3}{2}%
		\cdot \frac{9}{80}-\frac{1}{16}\cdot \frac{9}{80}=\frac{207}{1280} \\
		&\leq &\psi \left( M(x,y,d,T,\varphi )-\phi \left( M(x,y,d,T,\varphi
		)\right) \right) .
	\end{eqnarray*}%
	If 
	\begin{equation*}
		M(x,y,d,T,\varphi )=d(y,0)+y,
	\end{equation*}%
	then 
	\begin{equation*}
		d(y,0)+y\geq \frac{9}{80},
	\end{equation*}%
	since $x<y$ and $0<y.$ Thus 
	\begin{equation*}
		\psi \left( s^{2}d(Tx,Ty)+\varphi (Tx)+\varphi (Ty)\right) \leq \psi \left(
		M(x,y,d,T,\varphi )-\phi \left( M(x,y,d,T,\varphi )\right) \right) .
	\end{equation*}
	
	Case II: $x\in \left\{ 0,\frac{1}{5},\frac{1}{9},\frac{1}{16}\right\} $ and $%
	y\in \left[ \frac{1}{2},1\right] $. This implies $x<y$. Then 
	\begin{equation*}
		\psi \left( s^{2}\left( d\left( Tx,Ty\right) +\varphi (Tx)+\varphi
		(Ty)\right) \right) =\frac{3}{2}\left[ 9d\left( \frac{1}{16},0\right) +\frac{%
			1}{16}\right] =\frac{123}{160}.
	\end{equation*}%
	Also 
	\begin{equation*}
		d(x,y)+\varphi (x)+\varphi (y)=(x-y)^{2}+x+y,
	\end{equation*}%
	\begin{equation*}
		d(x,Tx)+\varphi (x)+\varphi (Tx)=d(x,0)+x
	\end{equation*}%
	\begin{equation*}
		d(y,Ty)+\varphi (y)+\varphi (Ty)=d\left( y,\frac{1}{16}\right) +y+\frac{1}{16%
		},
	\end{equation*}%
	and 
	\begin{equation*}
		M(x,y,d,T,\varphi )=\max \{(x-y)^{2}+x+y,d(x,0)+x,d\left( y,\frac{1}{16}%
		\right) +y+\frac{1}{16}\}
	\end{equation*}%
	Since $x<y$, we have 
	\begin{equation*}
		M(x,y,d,T,\varphi )=\max \{(x-y)^{2}+x+y,d\left( y,\frac{1}{16}\right) +y+%
		\frac{1}{16}\}.
	\end{equation*}%
	If 
	\begin{equation*}
		M(x,y,d,T,\varphi )=d(x-y)^{2}+x+y\geq \frac{1}{2}+\left( \frac{1}{2}-\frac{1%
		}{5}\right) ^{2}=\frac{59}{100},
	\end{equation*}%
	then 
	\begin{equation*}
		\psi \left( M(x,y,d,T,\varphi )\right) -\phi \left( M(x,y,d,T,\varphi
		)\right) =\frac{23}{16}\left( d(x,y)^{2}+x+y\right) \geq \frac{23}{16}\cdot 
		\frac{59}{100}=\frac{1357}{1600}\geq \frac{123}{160}.
	\end{equation*}%
	Then 
	\begin{equation*}
		\psi \left( s^{2}d(Tx,Ty)+\varphi (Tx)+\varphi (Ty)\right) \leq \psi \left(
		M(x,y,d,T,\varphi )-\phi \left( M(x,y,d,T,\varphi )\right) \right) .
	\end{equation*}%
	If 
	\begin{equation*}
		M(x,y,d,T,\varphi )=d\left( y,\frac{1}{16}\right) +y+\frac{1}{16}\geq \frac{1%
		}{2}+\frac{1}{16}+\left( \frac{1}{2}-\frac{1}{16}\right) ^{2}=\frac{193}{256}%
		,
	\end{equation*}%
	then 
	\begin{equation*}
		\psi \left( M(x,y,d,T,\varphi )\right) -\phi \left( M(x,y,d,T,\varphi
		)\right) =\frac{23}{16}d\left( y,\frac{1}{16}\right) +y+\frac{1}{16}\geq 
		\frac{23}{16}\cdot \frac{193}{256}=\frac{4439}{4096}\geq \frac{123}{160}.
	\end{equation*}%
	Then 
	\begin{equation*}
		\psi \left( s^{2}d(Tx,Ty)+\varphi (Tx)+\varphi (Ty)\right) \leq \psi \left(
		M(x,y,d,T,\varphi )-\phi \left( M(x,y,d,T,\varphi )\right) \right) .
	\end{equation*}
	
	Case III: $y\in \left\{ 0,\frac{1}{5},\frac{1}{9},\frac{1}{16}\right\} $ and 
	$x\in \left[ \frac{1}{2},1\right] .$
	
	By a similar method to Case II, we deduce that 
	\begin{equation*}
		\psi \left( s^{2}d(Tx,Ty)+\varphi (Tx)+\varphi (Ty)\right) \leq \psi \left(
		M(x,y,d,T,\varphi )-\phi \left( M(x,y,d,T,\varphi )\right) \right) .
	\end{equation*}
	
	Case IV: $x,y\in \left[ \frac{1}{2},1\right] $.
	
	If $x\geq y$, then 
	\begin{equation*}
		\psi \left( s^{2}d(Tx,Ty)+\varphi (Tx)+\varphi (Ty)\right) =\psi \left(
		9d\left( \frac{1}{16},\frac{1}{16}\right) +\varphi \left( \frac{1}{16}%
		\right) +\varphi \left( \frac{1}{16}\right) \right) =\frac{3}{16}.
	\end{equation*}%
	Also 
	\begin{equation*}
		d(x,y)+\varphi (x)+\varphi (y)=\left( x-y\right) ^{2}+x+y,
	\end{equation*}%
	\begin{equation*}
		d(x,Tx)+\varphi (x)+\varphi (Tx)=d\left( x,\frac{1}{16}\right) +x+\frac{1}{16%
		},
	\end{equation*}%
	\begin{equation*}
		d(y,Ty)+\varphi (y)+\varphi (Ty)=d\left( y,\frac{1}{16}\right) +y+\frac{1}{16%
		}
	\end{equation*}%
	and 
	\begin{equation*}
		M(x,y,d,T,\varphi )=\max \{\left( x-y\right) ^{2}+x+y,d\left( x,\frac{1}{16}%
		\right) +x+\frac{1}{16},d\left( y,\frac{1}{16}\right) +y+\frac{1}{16}\}.
	\end{equation*}%
	Since $x\geq y$, we have 
	\begin{equation*}
		M(x,y,d,T,\varphi )=\max \{\left( x-y\right) ^{2}+x+y,d\left( x,\frac{1}{16}%
		\right) +x+\frac{1}{16}\}.
	\end{equation*}%
	If 
	\begin{equation*}
		M(x,y,d,T,\varphi )=\left( x-y\right) ^{2}+x+y\geq 1,
	\end{equation*}%
	then 
	\begin{equation*}
		\frac{3}{16}\leq \frac{23}{16}\leq \psi \left( M(x,y,d,T,\varphi )-\phi
		\left( M(x,y,d,T,\varphi )\right) \right) .
	\end{equation*}%
	If 
	\begin{equation*}
		M(x,y,d,T,\varphi )=d\left( x,\frac{1}{16}\right) +x+\frac{1}{16}\geq
		d\left( \frac{1}{2},\frac{1}{16}\right) +\frac{1}{2}+\frac{1}{16}=\frac{193}{%
			256},
	\end{equation*}%
	then 
	\begin{equation*}
		\frac{3}{16}\leq \frac{23}{16}\cdot \frac{193}{256}=\frac{4439}{4096}\leq
		\psi \left( M(x,y,d,T,\varphi )-\phi \left( M(x,y,d,T,\varphi )\right)
		\right) .
	\end{equation*}%
	If $x,y\in A$ and $x<y$, then 
	\begin{equation*}
		M(x,y,d,T,\varphi )=\max \{\left( x-y\right) ^{2}+x+y,d\left( y,\frac{1}{16}%
		\right) +y\}.
	\end{equation*}%
	By a similar method to the condition $x\geq y$, we have 
	\begin{equation*}
		\frac{3}{16}\leq \frac{4439}{4096}\leq \psi \left( M(x,y,d,T,\varphi )-\phi
		\left( M(x,y,d,T,\varphi )\right) \right) .
	\end{equation*}%
	Hence 
	\begin{equation}
		\psi \left( s^{2}d\left( Tx,Ty\right) +\varphi \left( Tx\right) +\varphi
		\left( Ty\right) \right) \leq F(\psi \left( M\left( x,y,d,T,\varphi \right)
		),\phi \left( M\left( x,y,d,T,\varphi \right) \right) \right) ,
	\end{equation}%
	Thus all the conditions of Theorem \ref{thm3.2} are satisfied and 0 is the
	unique fixed point of $T$.
\end{example}

\begin{definition}
	\cite{aha}A tripled $(\psi ,\varphi ,F)$ where $\psi \in \Psi ,$ $\varphi
	\in \Phi _{u}$ and $F\in \mathcal{C}$ is say to be monotone if for any $%
	x,y\in \left[ 0,\infty \right) $%
	\begin{equation*}
		x\leqslant y\Longrightarrow F(\psi (x),\varphi (x))\leqslant F(\psi
		(y),\varphi (y)).
	\end{equation*}
\end{definition}

\begin{example}
	\cite{aha}let $F(s,t)=s-t,\phi (x)=\sqrt{x}$\textrm{%
		\begin{equation*}
			\psi (x)=%
			\begin{cases}
				\sqrt{x} & \text{if }0\leq x\leq 1, \\ 
				x^{2}, & \text{if x}>1%
			\end{cases}%
			,
		\end{equation*}%
	}then $(\psi ,\phi ,F)$ is monotone.
\end{example}

\begin{example}
	\cite{aha}let $F(s,t)=s-t,\phi (x)=x^{2}$\textrm{%
		\begin{equation*}
			\psi (x)=%
			\begin{cases}
				\sqrt{x} & \text{if }0\leq x\leq 1, \\ 
				x^{2}, & \text{if x}>1%
			\end{cases}%
			,
		\end{equation*}%
	}then $(\psi ,\phi ,F)$ is not monotone.
\end{example}

\begin{definition}
	\label{3.1+2022} Let $X$ be a complete $b$-rectangular metric space with
	metric $d$ and parameter $s$ and $T:X\rightarrow X$. Also let $\varphi
	:X\rightarrow R^{+}$ be a lower semicontinuous function. Then $T$ is called
	a convex $C$ class $-F-$generalized weakly contractive mapping if it
	satisfies the following condition: 
	\begin{equation}
		\psi \left( s^{2}d\left( Tx,Ty\right) +\varphi \left( Tx\right) +\varphi
		\left( Ty\right) \right) \leq F(\psi \left( M\left( x,y,d,T,\varphi \right)
		),\phi \left( M\left( x,y,d,T,\varphi \right) \right) \right) ,
	\end{equation}%
	where 
	\begin{equation*}
		M(x,y,d,T,\varphi )=\frac{a[d(x,y)+\varphi (x)+\varphi
			(y)]+b[d(x,Tx)+\varphi (x)+\varphi (Tx)]+c[d(y,Ty)+\varphi (y)+\varphi (Ty)]%
		}{a+b+c}
	\end{equation*}%
	for all $x,y\in X$,$a,b,c\in \lbrack 0,\infty )$,$a+b+c>0$ \ and ,$(\psi
	,\phi ,F)$ is not monotone ,$\ \ \psi \in \Phi ,\phi \in \Phi _{u}$, $\ F\in 
	\mathcal{C}$ .
\end{definition}

\begin{theorem}
	\label{thm3.2 +2022+9+6+1} Let X be a complete $b$-rectangular metric space
	with parameter s $\geq 1$. If $T$ is a is called a convex $C$ class $-F-$%
	generalized weakly contractive mapping, then $T$ has a unique fixed point $%
	z\in X$ such that $z=Tz$ and $\varphi (z)=0$.
\end{theorem}

\begin{proof}
	Let $x_{0}\in X$ be an arbitrary point in $X $. Then we define the sequence $%
	\left\lbrace x_{n}\right\rbrace$ by $x_{n+1} =Tx_{n}$, for all $n\in \mathbb{%
		\ N}.$
	
	If there exists $n_0\in \mathbb{N}$ such that $x_{n_0}=x_{n_0+1} =0$, then $%
	x_{n_0} $ is a fixed point of $T $.
	
	Next, we assume that $x_{n}\neq x_{n+1}$.\newline
	We claim that 
	\begin{equation*}
		\lim_{n\rightarrow +\infty }d\left( x_{n,}x_{n+1}\right) =0
	\end{equation*}%
	and 
	\begin{equation*}
		\lim_{n\rightarrow +\infty }d\left( x_{n,}x_{n+2}\right) =0.
	\end{equation*}%
	Letting $x=x_{n-1}$ and $y=x_{n}$ in (\ref{3.1}) for all $n\in $ $\mathbb{N}$
	, we have 
	\begin{eqnarray}
		&&\psi \left( s^{2}d\left( Tx_{n-1},Tx_{n}\right) +\varphi \left(
		Tx_{n-1}\right) +\varphi \left( Tx_{n}\right) \right)  \\
		&&\qquad \leq F(\psi \left( M\left( x_{n-1},x_{n},d,T,\varphi \right)
		\right) ,\phi \left( M\left( x_{n-1},x_{n},d,T,\varphi \right) \right) ), 
		\notag
	\end{eqnarray}%
	where 
	\begin{multline}
		M\left( x_{n-1},x_{n},d,T,\varphi \right) =\frac{1}{a+b+c}%
		\{[ad(x_{n-1},x_{n})+\varphi (x_{n-1})+\varphi
		(x_{n})]+b[d(x_{n-1},x_{n})+\varphi (x_{n-1})]  \label{n1} \\
		+\varphi (x_{n})]+c[d(x_{n},Tx_{n})+\varphi (x_{n})+\varphi (Tx_{n})]\} 
		\notag \\
		=\frac{1}{a+b+c}\{(a+b)[d(x_{n-1},x_{n})+\varphi (x_{n-1})+\varphi
		(x_{n})]+c[d(x_{n},x_{n+1})+\varphi (x_{n})+\varphi (Tx_{n+1})]\}.  \notag
	\end{multline}%
	then we have 
	\begin{align*}
		\psi \left( d\left( Tx_{n-1},Tx_{n}\right) +\varphi \left( Tx_{n-1}\right)
		+\varphi \left( x_{n+1}\right) \right) & =\psi \left( d\left(
		x_{n},x_{n+1}\right) +\varphi \left( x_{n}\right) +\varphi \left(
		x_{n+1}\right) \right)  \\
		& \leq \psi \left( s^{2}d\left( x_{n},x_{n+1}\right) +\varphi \left(
		x_{n}\right) +\varphi \left( x_{n+1}\right) \right)  \\
		& \leq F(\psi \left( 
		\begin{array}{c}
			\frac{1}{a+b+c}\{(a+b)[d(x_{n-1},x_{n})+\varphi (x_{n-1})+\varphi (x_{n})]
			\\ 
			+c[d(x_{n},x_{n+1})+\varphi (x_{n})+\varphi (Tx_{n+1})]\}%
		\end{array}%
		\right)  \\
		& ,\phi \left( 
		\begin{array}{c}
			\frac{1}{a+b+c}\{(a+b)[d(x_{n-1},x_{n})+\varphi (x_{n-1})+\varphi (x_{n})]
			\\ 
			+c[d(x_{n},x_{n+1})+\varphi (x_{n})+\varphi (Tx_{n+1})]\}%
		\end{array}%
		\right) ) \\
		& \leq \psi \left( 
		\begin{array}{c}
			\frac{1}{a+b+c}\{(a+b)[d(x_{n-1},x_{n})+\varphi (x_{n-1})+\varphi (x_{n})]
			\\ 
			+c[d(x_{n},x_{n+1})+\varphi (x_{n})+\varphi (Tx_{n+1})]\}%
		\end{array}%
		\right) ,
	\end{align*}%
	which implies 
	\begin{equation*}
		\phi \left( d\left( x_{n},x_{n+1}\right) +\varphi \left( x_{n}\right)
		+\varphi \left( x_{n+1}\right) \right) \leq \left( 
		\begin{array}{c}
			\frac{1}{a+b+c}\{(a+b)[d(x_{n-1},x_{n})+\varphi (x_{n-1})+\varphi (x_{n})]
			\\ 
			+c[d(x_{n},x_{n+1})+\varphi (x_{n})+\varphi (Tx_{n+1})]\}%
		\end{array}%
		\right) 
	\end{equation*}%
	and so 
	\begin{equation*}
		d\left( x_{n},x_{n+1}\right) +\varphi \left( x_{n}\right) +\varphi \left(
		x_{n+1}\right) \leq d(x_{n-1},x_{n})+\varphi (x_{n-1})+\varphi (x_{n})
	\end{equation*}%
	Hence  the sequence $\{d\left( x_{n},x_{n+1}\right) +\varphi \left(
	x_{n}\right) +\varphi \left( x_{n+1}\right) \}_{n\in \mathbb{N}}$ is
	nonincreasing.\newline
	Hence $d\left( x_{n},x_{n+1}\right) +\varphi \left( x_{n}\right) +\varphi
	\left( x_{n+1}\right) \rightarrow r$ as $n\rightarrow +\infty $ for some $%
	r\geq 0$. Assume $r>0$ and letting $n\rightarrow +\infty $ in (\ref{n1}) and
	using the continuity of $\psi $ and the lower semicontinuity of $\phi $, we
	have 
	\begin{align*}
		\psi \left( r\right) & \leq \psi \left( s^{2}r\right) \leq F(\psi \left(
		r\right) ,\liminf_{n\rightarrow \infty }\phi \left( d\left(
		x_{n},x_{n+1}\right) +\varphi \left( x_{n}\right) +\varphi \left(
		x_{n+1}\right) \right) ) \\
		& \leq F(\psi \left( r\right) ,\phi \left( r\right) ).
	\end{align*}%
	It follows that $\psi \left( r\right) =0,$ $\ $or$\ \ \ ,\phi \left(
	r\right) =0$, hence we have $r=0$ and consequently, $\lim_{n\rightarrow
		+\infty }d\left( x_{n},x_{n+1}\right) +\varphi \left( x_{n}\right) +\varphi
	\left( x_{n+1}\right) =0.$ So 
	\begin{equation}
		\lim_{n\rightarrow +\infty }d\left( x_{n},x_{n+1}\right) =0,
	\end{equation}%
	\begin{equation}
		\lim_{n\rightarrow +\infty }\varphi \left( x_{n}\right) =\lim_{n\rightarrow
			+\infty }\varphi \left( x_{n+1}\right) =0.
	\end{equation}
	
	Now, we shall prove that $T $ has a periodic point. Suppose that it is not
	the case. Then $x_{n}\neq x_{m}$ for all $n,m\in $ $\mathbb{N},\ n\neq m.$
	
	In (\ref{3.1}), letting $x=x_{n-1}$ and $y=x_{n+1}$, we have 
	\begin{eqnarray*}
		&&\psi \left( s^{2}d\left( Tx_{n-1},Tx_{n+1}\right) +\varphi \left(
		Tx_{n-1}\right) +\varphi \left( Tx_{n+1}\right) \right)  \\
		&&\qquad \leq F(\psi \left( M\left( x_{n-1},x_{n+1},d,T,\varphi \right)
		\right) ,\phi \left( M\left( x_{n-1},x_{n+1},d,T,\varphi \right) \right) ),
	\end{eqnarray*}%
	where 
	\begin{eqnarray*}
		M\left( x_{n-1},x_{n+1},d,T,\varphi \right)  &=&\frac{1}{a+b+c}%
		\{a[d(x_{n-1},x_{n+1})+\varphi (x_{n-1})+\varphi (x_{n+1})]+ \\
		&&b[d(x_{n-1},x_{n})+\varphi (x_{n-1})+\varphi
		(x_{n})]+c[d(x_{n+1},x_{n+2})+\varphi (x_{n+1})+\varphi (x_{n+2})]\} \\
		&=&\frac{1}{a+b+c}\{a[d(x_{n-1},x_{n+1})+\varphi (x_{n-1})+\varphi
		(x_{n+1})],d(x_{n-1},x_{n})+\varphi (x_{n-1})+\varphi (x_{n})\}.
	\end{eqnarray*}%
	So we get 
	\begin{eqnarray}
		&&\psi \left( d\left( x_{n},x_{n+2}\right) +\varphi \left( x_{n}\right)
		+\varphi \left( x_{n+2}\right) \right) \leq \psi \left( s^{2}d\left(
		x_{n},x_{n+2}\right) +\varphi \left( x_{n}\right) +\varphi \left(
		x_{n+2}\right) \right)   \label{3.11} \\
		&\leq &F(\psi \left( \max \{d(x_{n-1},x_{n+1})+\varphi (x_{n-1})+\varphi
		(x_{n+1}),d(x_{n-1},x_{n})+\varphi (x_{n-1})+\varphi (x_{n})\}\right)  
		\notag \\
		&&,\phi \left( \max \{d(x_{n-1},x_{n+1})+\varphi (x_{n-1})+\varphi
		(x_{n+1}),d(x_{n-1},x_{n})+\varphi (x_{n-1})+\varphi (x_{n})\}\right) ). 
		\notag
	\end{eqnarray}%
	Take $a_{n}=d\left( x_{n},x_{n+2}\right) +\varphi \left( x_{n}\right)
	+\varphi \left( x_{n+2}\right) $ and $b_{n}=d(x_{n},x_{n+1})+\varphi
	(x_{n})+\varphi (x_{n+1}).$\newline
	Then by (\ref{3.11}), one can write 
	\begin{align*}
		\psi \left( a_{n}\right) & \leq F(\psi \left( \max \left(
		a_{n-1},b_{n-1}\right) \right) ,\phi \left( \max \left(
		a_{n-1},b_{n-1}\right) \right) ) \\
		& \leq \psi \left( \max \left( a_{n-1},b_{n-1}\right) \right) .
	\end{align*}%
	Since $\psi $ is increasing, we get 
	\begin{equation*}
		a_{n}\leq \max \left\{ a_{n-1},b_{n-1}\right\} . \\
	\end{equation*}%
	By (\ref{3.5}), we have 
	\begin{equation*}
		b_{n}\leq b_{n-1}\leq \max \left\{ a_{n-1},b_{n-1}\right\} ,
	\end{equation*}%
	which implies that 
	\begin{equation*}
		\max \left\{ a_{n},b_{n}\right\} \leq \max \left\{ a_{n-1},b_{n-1}\right\} ,%
		\text{ }\forall n\in \mathbb{N}. \\
	\end{equation*}%
	Therefore, the sequence $\max \left\{ a_{n-1},b_{n-1}\right\} _{n\in \mathbb{%
			\ N}}$ is a nonnegative decreasing sequence of real numbers. Thus there
	exists $\lambda \geq 0$ such that 
	\begin{equation*}
		\lim_{n\rightarrow +\infty }\max \left\{ a_{n},b_{n}\right\} =\lambda .
	\end{equation*}%
	Assume that $\lambda >0$. By (\ref{3.8}), it is obvious that 
	\begin{equation}
		\lambda =\lim_{n\rightarrow +\infty }\sup a_{n}=\lim_{n\rightarrow +\infty
		}\sup \max \left\{ a_{n},b_{n}\right\} =\lim_{n\rightarrow +\infty }\max
		\left\{ a_{n},b_{n}\right\} .  \label{3.12}
	\end{equation}%
	Taking $\limsup_{n}\rightarrow +\infty $ in (\ref{3.11}), using (\ref{3.12})
	and using the properties of $\psi $ and $\phi $, we obtain 
	\begin{align*}
		\psi (\lambda )& =\psi \left( \limsup_{n\rightarrow +\infty }a_{n}\right)  \\
		& =\limsup_{n\rightarrow +\infty }\psi \left( a_{n}\right)  \\
		& \leq \limsup_{n\rightarrow +\infty }\psi \left( \max \left\{
		a_{n},b_{n}\right\} \right) -\liminf_{n\rightarrow +\infty }\phi \left( \max
		\left\{ a_{n},b_{n}\right\} \right)  \\
		& \leq F(\psi \left( \lim_{n\rightarrow +\infty }\max \left\{
		a_{n},b_{n}\right\} \right) ,\phi \left( \lim_{n\rightarrow +\infty }\max
		\left\{ a_{n},b_{n}\right\} \right) ) \\
		& =F(\psi (\lambda ),\phi (\lambda )),
	\end{align*}%
	which implies that $\psi \left( r\right) =0,$ $\ $or$\ \ \ ,\phi \left(
	r\right) =0$, a contradiction. Thus, from (\ref{3.12}), 
	\begin{equation*}
		\limsup_{n\rightarrow +\infty }a_{n}=0
	\end{equation*}%
	and hence 
	\begin{equation*}
		\lim_{n\rightarrow +\infty }d\left( x_{n,}x_{n+2}\right) =0.
	\end{equation*}
	
	Next, we shall prove that $\left\lbrace x_{n}\right\rbrace _{n\in \mathbb{N}%
	} $ is a Cauchy sequence, i.e, $\lim_{n,m\rightarrow +\infty }d\left(
	x_{n,}x_{m}\right) =0$ for all $n,m\in \mathbb{N}$. Suppose to the contrary.
	By Lemma $\ref{2?4}$, there is an $\varepsilon $ $>0$ such that for an
	integer $k $ there exist two sequences $\left\lbrace n_{\left( k\right)
	}\right\rbrace $ and $\left\lbrace m_{\left( k\right) }\right\rbrace $ such
	that
	
	\begin{itemize}
		\item[i)] $\varepsilon \leq \lim_{k\rightarrow +\infty }\inf d\left(
		x_{m_{\left( k\right) }},x_{n_{\left( k\right)}}\right) \leq
		\lim_{k\rightarrow +\infty }\sup d\left( x_{m_{\left( k\right)
		}},x_{n_{\left( k\right)}}\right)\leq s\varepsilon ,$
		
		\item[ii)] $\varepsilon \leq \lim_{k\rightarrow +\infty }\inf d\left(
		x_{n_{\left( k\right) }},x_{m_{\left( k\right)}+1}\right) \leq
		\lim_{k\rightarrow +\infty }\sup d\left( x_{n_{\left( k\right)
		}},x_{m_{\left( k\right)}+1}\right)\leq s\varepsilon ,$
		
		\item[iii)] $\varepsilon \leq \lim_{k\rightarrow +\infty }\inf d\left(
		x_{m_{\left( k\right) }},x_{n_{\left( k\right)}+1}\right) \leq
		\lim_{k\rightarrow +\infty }\sup d\left( x_{m_{\left( k\right)
		}},x_{n_{\left( k\right)}+1}\right)\leq s\varepsilon ,$
		
		\item[vi)] $\frac{\varepsilon}{s} \leq \lim_{k\rightarrow +\infty }\inf
		d\left( x_{m_{\left( k\right) +1}},x_{n_{\left( k\right)}+1}\right) \leq
		\lim_{k\rightarrow +\infty }\sup d\left( x_{m_{\left(
				k\right)}+1},x_{n_{\left( k\right)}+1}\right)\leq s^{2}\varepsilon .$
	\end{itemize}
	
	From (\ref{3.1}) and by setting $x=x_{m_{\left( k\right) }}$ and $%
	y=x_{n_{\left( k\right) }}$, we have 
	\begin{eqnarray*}
		&&M(x_{m_{(k)}},x_{n_{(k)}},d,T,\varphi )=\max
		\{d(x_{m_{(k)}},x_{n_{(k)}})+\varphi (x_{m_{(k)}})+\varphi (x_{m_{(k)}}), \\
		&&\quad d(x_{m_{(k)}},x_{m_{(k)}+1})+\varphi (x_{m_{(k)}})+\varphi
		(x_{m_{(k)}+1}),d(x_{n_{(k)}},x_{n_{(k)}+1})+\varphi (x_{n_{(k)}})+\varphi
		(x_{n_{(k)}+1})\}.
	\end{eqnarray*}%
	Taking the limit as $k\rightarrow +\infty $ and using (\ref{3.8}), (\ref{3.9}
	) and $(iii)$ of Lemma \ref{2?4}, we have 
	\begin{equation}
		\lim_{k\rightarrow +\infty }M\left( x_{m_{\left( k\right) }},x_{n_{\left(
				k\right) }},d,T,\varphi \right) \leq s\varepsilon .  \label{3.15}
	\end{equation}%
	Now letting $x=x_{m_{\left( k\right) }}$ and $y=x_{n_{\left( k\right) }}$ in
	(\ref{3.1}), we have 
	\begin{multline*}
		\psi \left[ s^{2}d\left( x_{m_{\left( k\right) }+1},x_{n_{\left( k\right)
			}+1}\right) +\varphi \left( m_{\left( k\right) }+1\right) +\varphi \left(
		n_{\left( k\right) }+1\right) \right] \\
		\leq F(\psi \left[ d\left( x_{m_{\left( k\right) }+1},x_{n_{\left( k\right)
			}+1}\right) +\varphi \left( m_{\left( k\right) }+1\right) +\varphi \left(
		n_{\left( k\right) }+1\right) \right] \\
		,\phi \left[ d\left( x_{m_{\left( k\right) }},x_{n_{\left( k\right)
			}+1}\right) +\varphi \left( m_{\left( k\right) }\right) +\varphi \left(
		n_{\left( k\right) }\right) \right] ).
	\end{multline*}%
	Letting $k\rightarrow +\infty $, using (\ref{3.8}), (\ref{3.9}), (\ref{3.15}
	), and applying the continuity of $\psi $ and the lower semicontinuity of $%
	\phi $, we have 
	\begin{equation*}
		\lim_{k\rightarrow +\infty }\psi \left[ s^{2}d\left( x_{m_{\left( k\right)
			}+1},x_{n_{\left( k\right) }+1}\right) \right] \leq F(\psi (s\varepsilon
		),\phi (s\varepsilon )).
	\end{equation*}%
	Using (\ref{3.15}) and $(iv)$ of Lemma $\ref{2?4}$, we obtain 
	\begin{equation*}
		\psi (s\varepsilon )=\psi \left( s^{2}\frac{\varepsilon }{s}\right) \leq
		\lim_{k\rightarrow +\infty }\sup \psi \left[ s^{2}d\left( x_{m_{\left(
				k\right) }+1},x_{n_{\left( k\right) }+1}\right) \right] \leq F(\psi
		(s\varepsilon ),\phi (s\varepsilon )).
	\end{equation*}%
	This is a contradiction. Thus 
	\begin{equation*}
		\lim_{n,m\rightarrow +\infty }d\left( x_{m},x_{n}\right) =0.
	\end{equation*}%
	Hence $\left\{ x_{n}\right\} $ is a Cauchy sequence in $X$. By completeness
	of $\left( X,d\right) $, there exists $z\in X$ such that 
	\begin{equation*}
		\lim_{n\rightarrow +\infty }d\left( x_{n},z\right) =0.
	\end{equation*}%
	Since $\varphi $ is lower semicontinuous, we get 
	\begin{equation*}
		\varphi \left( z\right) \leq \liminf_{n\rightarrow +\infty }\varphi \left(
		x_{n}\right) \leq \lim_{n\rightarrow +\infty }\varphi \left( x_{n}\right) =0,
		\\
	\end{equation*}%
	which implies 
	\begin{eqnarray}
		\varphi \left( z\right) &=&0.  \label{3.17} \\
		&&  \notag
	\end{eqnarray}%
	Now, putting $x=x_{n}$ and $y=z$ in (\ref{3.1}), we have 
	\begin{eqnarray*}
		M\left( x_{n},z,d,T,\varphi \right) &=&\max \{d\left( x_{n},z\right)
		+\varphi \left( x_{n}\right) +\varphi \left( z\right) , \\
		&&d\left( x_{n},x_{n+1}\right) +\varphi \left( x_{n}\right) +\varphi \left(
		x_{n+1}\right) ,d\left( z,Tz\right) +\varphi \left( z\right) +\varphi \left(
		Tz\right) \}.
	\end{eqnarray*}%
	Taking the limit as $n\rightarrow +\infty $ and using (\ref{3.8}), (\ref{3.9}
	) and (\ref{3.17}), we have 
	\begin{equation*}
		\lim_{n\rightarrow +\infty }M\left( x_{n},z,d,T,\varphi \right) =d\left(
		z,Tz\right) +\varphi \left( Tz\right) .
	\end{equation*}%
	Since $x_{n}\rightarrow z$ as $n\rightarrow +\infty $, from Lemma \ref{2.3}
	, we conclude that 
	\begin{equation*}
		\frac{1}{s}d\left( z,Tz\right) \leq \lim_{n\rightarrow +\infty }\sup d\left(
		Tx_{n},Tz\right) \leq sd\left( z,Tz\right) .
	\end{equation*}%
	Hence 
	\begin{equation*}
		sd\left( z,Tz\right) =s^{2}\frac{1}{s}d\left( z,Tz\right) \leq
		\lim_{n\rightarrow +\infty }\sup s^{2}d\left( Tx_{n},Tz\right) ,
	\end{equation*}%
	which implies 
	\begin{equation*}
		\lim_{n\rightarrow +\infty }\sup \left[ sd\left( z,Tz\right) +\varphi \left(
		x_{n+1}\right) +\varphi \left( Tz\right) \right] \leq \lim_{n\rightarrow
			+\infty }\sup \left[ s^{2}d\left( Tx_{n},Tz\right) +\left( x_{n+1}\right)
		+\varphi \left( Tz\right) \right] .
	\end{equation*}%
	Then using (\ref{3.1}), we have 
	\begin{align*}
		\psi \left[ s^{2}d\left( Tx_{n},Tz\right) +\varphi \left( Tx_{n}\right)
		+\varphi \left( Tz\right) \right] & =\psi \left[ s^{2}d\left(
		x_{n+1},Tz\right) +\varphi \left( x_{n+1}\right) +\varphi \left( Tz\right) %
		\right] \\
		& \leq F(\psi \left[ M\left( x_{n},z,d,T,\varphi \right) \right] ,\phi \left[
		M\left( x_{n},z,d,T,\varphi \right) \right] ).
	\end{align*}%
	Letting $n\rightarrow +\infty $ and using the continuity of $\psi $ and the
	lower semicontinuity of $\phi $, we have 
	\begin{multline*}
		\psi \left[ \lim_{n\rightarrow +\infty }\sup \left( sd\left( z,Tz\right)
		+\varphi \left( x_{n+1}\right) +\varphi \left( Tz\right) \right) \right] \\
		\leq \psi \left[ \lim_{n\rightarrow +\infty }\sup \left( s^{2}d\left(
		Tx_{n},Tz\right) +\left( x_{n+1}\right) +\varphi \left( Tz\right) \right) %
		\right] \\
		\leq F(\psi \left[ \lim_{n\rightarrow +\infty }\sup M\left(
		x_{n},z,d,T,\varphi \right) \right] ,\lim_{n\rightarrow +\infty }\phi \left[
		M\left( x_{n},z,d,T,\varphi \right) \right] ),
	\end{multline*}%
	which implies 
	\begin{equation*}
		\psi \left[ sd\left( z,Tz\right) +\varphi \left( Tz\right) \right] \leq
		F(\psi \left[ d\left( z,Tz\right) +\varphi \left( Tz\right) \right] ,\phi %
		\left[ d\left( z,Tz\right) +\varphi \left( Tz\right) \right] ).
	\end{equation*}%
	This holds if and only if \ $\psi \left( d\left( z,Tz\right) +\varphi \left(
	Tz\right) \right) =0$\ \ or $\phi \left( d\left( z,Tz\right) +\varphi \left(
	Tz\right) \right) =0$ and from the property of $\psi ,\phi $, we have 
	\begin{equation*}
		d\left( z,Tz\right) +\varphi \left( Tz\right) =0.
	\end{equation*}%
	Hence $d\left( z,Tz\right) =0$ and so $z=Tz$ and $\varphi \left( Tz\right)
	=0 $. It is a contradiction to the assumption: that $T$ does not have a
	periodic point. Thus $T$ has a periodic point, say, $z$ of period $n$.
	Suppose that the set of fixed points of $T$ is empty. Then we have 
	\begin{equation*}
		q>0\ and\ d(z,Tz)>0.
	\end{equation*}%
	Since $T$ has a periodic point, $z=T^{n}z$. Letting $x=T^{n-1}z$ and $%
	y=T^{n}z$, we obtain 
	\begin{multline*}
		M\left( T^{n}z,T^{n-1}z,d,T,\varphi \right) =\max \{d\left(
		T^{n-1}z,T^{n}z\right) +\varphi \left( T^{n-1}z\right) +\varphi \left(
		T^{n}z\right) , \\
		d\left( T^{n-1}z,T^{n}z\right) +\varphi \left( T^{n-1}z\right) +\varphi
		\left( T^{n}z\right) ,d\left( T^{n}z,TT^{n}z\right) +\varphi \left(
		T^{n}z\right) +\varphi \left( TT^{n}z\right) \}.
	\end{multline*}
	
	By a similar method to (\ref{3.6}), we conclude that 
	\begin{equation*}
		M\left( T^{n}z,T^{n-1}z,d,T,\varphi \right) =d\left( T^{n-1}z,T^{n}z\right)
		+\varphi \left( T^{n-1}z\right) +\varphi \left( T^{n}z\right) .
	\end{equation*}
	From (\ref{3.1}), we have 
	\begin{align*}
		\psi \left[ s^{2}d\left( z,Tz\right) +\varphi \left( T^{n}z\right) +\varphi
		\left( T^{n+1}z\right) \right] & =\psi \left[ s^{2}d\left(
		T^{n}z,T^{n+1}z\right) +\varphi \left( T^{n}z\right) +\varphi \left(
		T^{n+1}z\right) \right] \\
		& \leq F(\psi \left[ d\left( T^{n-1}z,T^{n}z\right) +\varphi \left(
		T^{n-1}z\right) +\varphi \left( T^{n}z\right) \right] \\
		& ,\phi \left[ d\left( T^{n-1}z,T^{n}z\right) +\varphi \left(
		T^{n-1}z\right) +\varphi \left( T^{n}z\right) \right] ) \\
		& \leq \psi \left[ s^{2}d\left( T^{n-1}z,T^{n}z\right) +\varphi \left(
		T^{n-1}z\right) +\varphi \left( T^{n}z\right) \right] \\
		& \vdots \\
		& \leq F(\psi \left[ d\left( z,Tz\right) +\varphi \left( z\right) +\varphi
		\left( Tz\right) \right] \\
		& ,\phi \left[ d\left( z,Tz\right) +\varphi \left( z\right) +\varphi \left(
		Tz\right) \right] )
	\end{align*}
	Taking the limit as $n\rightarrow +\infty $ and applying the continuity of $%
	\psi $ and the lower semicontinuity of $\phi $, we have 
	\begin{equation*}
		\psi \left[ s^{2}d\left( z,Tz\right) \right] \leq F(\psi \left[ d\left(
		z,Tz\right) \right] ,\phi \left[ d\left( z,Tz\right) \right] ).
	\end{equation*}
	Hence $d(z,Tz)=0$, which is a contradiction. Thus the set of fixed points of
	T is non-empty, that is, $T$ has at least one fixed point.
	
	Suppose that $z,u\in X$ are two fixed points of $T$ such that $u\neq z$.
	Then $Tz=z $ and $Tu=u $.
	
	Letting $x=z$ and $y=u$ in (\ref{3.1}), we have 
	\begin{equation*}
		\psi \left( s^{2}d\left( Tz,Tu\right) +\varphi \left( Tz\right) +\varphi
		\left( Tu\right) \right) =\psi \left( s^{2}d\left( z,u\right) \right) \leq
		F(\psi \left( M\left( z,u,d,T,\varphi \right) \right) ,\phi \left( M\left(
		z,u,d,T,\varphi \right) \right) ),
	\end{equation*}%
	where 
	\begin{align*}
		M(z,u,d,T,\varphi )& =\max \{d(z,u)+\varphi (z)+\varphi (u),d(z,Tz)+\varphi
		(z)+\varphi (Tz),d(u,Tu)+\varphi (u)+\varphi (Tu)\} \\
		& =d(z,u).
	\end{align*}%
	So 
	\begin{equation*}
		\psi \left( s^{2}d\left( z,u\right) \right) \leq F(\psi \left( d\left(
		z,u\right) \right) ,\phi \left( d\left( z,u\right) \right) ).
	\end{equation*}%
	This holds if $\phi \left( d\left( z,u\right) \right) =0$ and so we have $%
	(d\left( z,u\right) =0$. Hence $z=u$ and $T$ has a unique fixed point.
\end{proof}

\section{Conclusion}

In this paper, inspired by the concept of generalized weakly contractive
mappings in metric spaces, we introduced the concept of C-class function for
generalized weakly contractive mappings in rectangular $b$-metric spaces to
study the existence of fixed point for the mappings in this spaces.
Furthermore, we provided some useful examples.

\medskip

\section*{Declarations}

\medskip

\noindent \textbf{Availablity of data and materials}\newline
\noindent Not applicable.

\medskip

\noindent \textbf{Competing interests}\newline
\noindent The authors declare that they have no competing interests.

\medskip

\noindent \textbf{Funding} \newline
\noindent Not applicable.

\medskip

\noindent \textbf{Authors' contributions}\newline
\noindent The authors equally conceived of the study, participated in its
design and coordination, drafted the manuscript, participated in the
sequence alignment, and read and approved the final manuscript.

\medskip

\bibliographystyle{amsplain}

\end{document}